\documentclass[12pt]{article}

\usepackage{amsmath,amssymb,amscd,amsthm,makeidx,txfonts,graphicx,url,bm}
\usepackage{a4wide,pstricks,xy}
\usepackage{mathtools,tikz-cd}

\usepackage[backref]{hyperref}

\numberwithin{equation}{subsection}

\xyoption{all}
\input{xypic}

\newtheorem{theorem}{Theorem}[section]
\newtheorem{proposition}[theorem]{Proposition}
\newtheorem{corollary}[theorem]{Corollary}
\newtheorem{conjecture}[theorem]{Conjecture}
\newtheorem{lemma}[theorem]{Lemma}

\theoremstyle{remark}
\newtheorem{remark}[theorem]{Remark}

\theoremstyle{definition}
\newtheorem{definition}[theorem]{Definition}

\def\beq{\begin{eqnarray}}
\def\eeq{\end{eqnarray}}
\def\bes{\begin{eqnarray*}}
\def\ees{\end{eqnarray*}}

\def\oF{\overline{\F}}

\DeclareMathOperator{\Aut}{Aut} 
 \DeclareMathOperator{\Hom}{Hom}

\def\C{\mathbb{C}}

\def\calA{{\mathcal{A}}}

\def\calU{{\mathcal{U}}}

\def\calF{{\mathcal{F}}}
\def\calH{{\mathcal{H}}}

\def\calP{\mathcal{P}}

\def\x{\mathbf{x}}

\def\P{\mathcal{P}}
\def\H{\mathbb{H}}

\def\N{\mathbb{Z}_{\geq 0}}

\def\F{\mathbb{F}}
\def\Q{\mathbb{Q}}

\def\calC{{\mathcal C}}
\def\calB{{\mathcal B}}
\def\calA{{\mathcal A}}

\def\Z{\mathbb{Z}}
\def\Gm{\mathbb{G}_m}

\def\diag{{\rm diag}}

\newcommand{\nc}{\newcommand}

\def\E{{\rm E}}

\nc{\op}[1]{\mathop{\mathchoice{\mbox{\rm #1}}{\mbox{\rm #1}}
{\mbox{\rm \scriptsize #1}}{\mbox{\rm \tiny #1}}}\nolimits}
\nc{\al}{\alpha}

\nc{\ep}{\varepsilon} \nc{\ga}{\gamma} \nc{\Ga}{\Gamma}
\nc{\la}{\lambda} \nc{\La}{\Lambda} \nc{\si}{\sigma}
\nc{\Sig}{{\Gamma}} \nc{\Om}{\Omega} \nc{\om}{\omega}
\nc{\Frob}{\op{ Frob}}

\nc{\SL}{{\rm SL}} \nc{\GL}{{\rm GL}} \nc{\PGL}{{\rm PGL}}
\nc{\G}{{\rm G}}

\def\calM{{\mathcal{M}}}
\def\x{{\bf x}}
\def\rmH{{\rm H}}

\nc{\cpt}{{\op{cpt}}} \nc{\Dol}{{\op{Dol}}} \nc{\DR}{{\op{DR}}}
\nc{\B}{{\op{B}}} \nc{\Triv}{\op{Triv}} \nc{\Hod}{{\op{Hod}}}
\nc{\Log}{{\op{Log}}} \nc{\Exp}{{\op{Exp}}} \nc{\Est}{E_{\op{st}}}
\nc{\Hst}{H_{\op{st}}} \nc{\Left}[1]{\hbox{$\left#1\vbox to
  10.5pt{}\right.\nulldelimiterspace=0pt \mathsurround=0pt$}}
\nc{\Right}[1]{\hbox{$\left.\vbox to
  10.5pt{}\right#1\nulldelimiterspace=0pt \mathsurround=0pt$}}
\nc{\LEFT}[1]{\hbox{$\left#1\vbox to
  15.5pt{}\right.\nulldelimiterspace=0pt \mathsurround=0pt$}}
\nc{\RIGHT}[1]{\hbox{$\left.\vbox to
 15.5pt{}\right#1\nulldelimiterspace=0pt \mathsurround=0pt$}}

\nc{\bee}{{\bf E}} \nc{\bphi}{{\bf \Phi}}

\begin{document}

\title{$E$-series of character varieties of non-orientable surfaces}

\author{Emmanuel Letellier \\ {\it
  Universit\'e de Paris} \\ {\tt
  letellier@math.univ-paris-diderot.fr}\and Fernando Rodriguez-Villegas
  \\ 
{\it  ICTP Trieste} \\ {\tt villegas@ictp.it}  }

 \pagestyle{myheadings}

\maketitle

\begin{abstract} In this paper we are interested in two kinds of
  (stacky) character varieties associated to a compact non-orientable
  surface   $\Sigma$. (A) We consider simply the quotient stack of the
  space of   representations of the fundamental group of $\Sigma$ to
  $\GL_n$. (B)   We choose a set $S$ of $k$-punctures of $\Sigma$ and
  a generic   $k$-tuple of semisimple conjugacy classes of $\GL_n$, and we
  consider the stack of anti-invariant local systems on the
  orientation cover of $\Sigma$ with local monodromies around the
  punctures given by the prescribed conjugacy classes. We compute the
  number of points of these spaces over finite fields from which we
  get a formula for their $E$-series (a certain specialization of the
  mixed Poincar\'e series). 
  Unexpectedly (see Remark~\ref{orientable}), when the Euler
  characteristic of $\Sigma$ is even, our formulas turn out to be
  closely related to those arising from the character varieties of
  punctured compact orientable Riemann surfaces studied in~\cite{HLRV}
  and~\cite{HLRV1}.
 \end{abstract}

\newpage
\tableofcontents
\newpage

\section{Introduction}

Let $K$ an algebraically closed field of characteristic $\neq 2$.

We let $r$ be a non-negative integer, put $\varrho=r-2$, and let
$\Sigma$ denote a non-orientable compact surface of Euler
characteristic $-\varrho$ (the connected sum of $r$ real projective
planes). Let $S=\{\alpha_1,\dots,\alpha_k\}$ be a set of $k$ points of
$\Sigma$. Fix a base point $b\in \Sigma\backslash S$. The fundamental
group $\Pi=\pi_1(\Sigma\backslash S,b)$ has the following standard
presentation

$$
\Pi=\left\langle a_1,\dots,a_r,x_1,\dots,x_k\,\left|\, a_1^2\cdots
    a_r^2x_1\cdots x_k=1\right\rangle\right. 
$$
Here $x_i$ represents a small loop encircling the $i$-th puncture.

Put $\G=\GL_n(K)$ and let $\sigma:\G\rightarrow \G$ be the Cartan  
involution $g\mapsto {^t}g^{-1}$. We let
$\G^+=\G\rtimes\langle\sigma\rangle$ be the corresponding
semi-direct product. Fix  a $k$-tuple $\calC=(C_1,\dots,C_k)$ of
conjugacy classes of $\G$ and for
$\varepsilon\in \langle\sigma\rangle$ consider the representation variety
$$
\Hom_\calC^\varepsilon(\Pi,\G^+):=\left\{\rho\in {\rm
    Hom}(\Pi,\G^+)\,\left|\, \pi\circ\rho(a_i)=\varepsilon,\,
    \rho(x_j)\in \iota(C_j)\text{ for all } i=1,\dots,r\text{ and
    }j=1,\dots,k\right\}\right. 
$$
where $\pi:\G^+\rightarrow \langle \sigma\rangle$ is the quotient map
and $\iota:G\rightarrow G^+$ the  natural  inclusion.

The conjugation action of $\G$ on $\G^+$ induces an action on ${\rm
  Hom}_\calC^\varepsilon(\Pi,\G^+)$ and we consider the quotient stack

$$
\calM_\calC^\varepsilon:=\left[\Hom_\calC^\varepsilon(\Pi,\G^+)/\G\right].
$$
If $\varepsilon=1$ then $\Hom_\calC^\varepsilon(\Pi,\G^+)$ is  the space 
\begin{align*}
\Hom_\calC(\Pi,\G)&=\left\{\rho\in\Hom(\Pi,\G)\,|\, \rho(x_i)\in C_i, i=1,\dots,k\right\}\\
&=\left.\left\{(A_1,\dots,A_r,X_1,\dots,X_k)\in \G^r\times
  \prod_{i=1}^k C_i\,\right|\, A_1^2\cdots A_r^2\, X_1\cdots
  X_k=1\right\}. 
\end{align*}
On the other hand, if $\varepsilon=\sigma$
then
$$
\Hom_\calC^\varepsilon(\Pi,\G^+)=\left.\left\{(A_1,\dots,
    A_r,X_1,\dots, X_k)\in \G^r\times\prod_{i=1}^kC_i\,\right|\,
  A_1\sigma(A_1)\cdots A_r\sigma(A_r)X_1\cdots X_k=1\right\}. 
$$
The last space has an interpretation in terms of anti-invariant
representations $\tilde{\Pi}\rightarrow \G$ of the fundamental group
$\tilde{\Pi}$ of $\tilde{\Sigma}$, where
$p:\tilde{\Sigma}\rightarrow\Sigma$ is the orientation covering (see
\S \ref{charactervar} for details). As it will become clear below,
this space is much better behaved  than the previous one.

In this paper we consider the following two cases :

(A) $\varepsilon=1$ and $S=\emptyset$.

(B) $\varepsilon=\sigma$, the $k$-tuple $\calC=(C_1,\dots, C_k)$ is
\emph{generic} and the $C_i$ are semisimple.


We prove that the corresponding stacks $\calM_\calC^\varepsilon$
are \emph{rational count}. In other words, the number of points over
generic finite fields is given by a fixed rational function in the
size of the field (see Definition~\ref{ratnl-count-defn} for the
formal definition). As in the case of polynomial count
(see~\cite[Appendix]{HRV}) the corresponding counting rational
function has a geometric meaning. Indeed, by Theorem~\ref{ratnl-count}
it coincides with the $E$-series
$E(\calM_\calC^\varepsilon;x)$,
which is the specialization $t\mapsto -1$
of the mixed Poincar\'e series
$H_c(\calM_\calC^\varepsilon;x,t)$.
This latter series encodes the dimension of the successive
subquotients of the weight filtration on the compactly supported
cohomology of $\calM_\calC^\varepsilon$
(see~\S \ref{MPS}). 


As we will see the genericity assumption on $\calC$
will simplify the calculation when $\varepsilon=\sigma$.
 It is not clear how to impose
genericity in the case $\varepsilon=1$.
For example, if $r=1=k$ then the equation
$$
A^2=\xi\, {\rm I}_n
$$
defining $\calM_\calC^\varepsilon$ for a given $\xi$ is just
equivalent to the case $\xi=1$ by replacing $A$ by $\sqrt{\xi}A$.

\subsection{$E$-series of $\calM_\calC^\varepsilon$ in case (A)}
\label{non-orient}
Put
$$
\calM_{\varrho,n}:=\left[\left\{(A_1,\dots,A_r)\in\G^r\,|\, A_1^2\cdots
    A_r^2={\rm I}_n\right\}/\G\right],  \qquad  \varrho:=r-2,
$$
and consider the generating function
\begin{equation}
\label{M-defn}
M_\varrho(q,T)=\sum_{n\geq 0} M_{\varrho,n}(q)\,T^n:=1+\sum_{n\geq
  1}q^{-\varrho\binom n2}\, |\calM_{\varrho,n}(\F_q)|\,T^n. 
\end{equation}

For $\lambda\in\calP$ denote by 
\begin{equation}
\label{hook-polyn}
H_\lambda(q)=\prod_x(q^{h(x)}-1)
\end{equation}
 the \emph{hook polynomial} (where $x$ runs over the set of boxes of the
Young diagram of $\lambda$ and $h(x)$ is the hook length).   


For an integer $\varrho$ consider the generating function
\begin{equation}
\label{Z-defn}
Z_\varrho(q,T):=\sum_{\lambda\in\calP}\left(q^{-n(\lambda)}H_\lambda(q)\right)^\varrho T^{|\lambda|},
\end{equation}
where $n(\lambda)=\sum_{i>0}(i-1)\lambda_i$ if
$\lambda=(\lambda_1,\lambda_2,\dots)\in\calP$.  Define 
$\{V_{\varrho,n}(q)\}_n$ by the formula
\begin{equation}
\label{V_n-defn}
\sum_{n\geq 1}V_{\varrho,n}(q) T^n:=\Log\,\left(Z_\varrho(q,T)\right)
\end{equation}
and, following~\cite{HMRV}, let for any positive  integer  $k$
\begin{equation}
\label{V_n-defn-1}
V_{\varrho,n,k}(q):=\sum_{\substack{\, \,
    m\mid k^\infty\\\! m|n}}\frac 1m V_{\varrho, \frac nm}(q^m),
\end{equation}
where $m\mid k^\infty$ means that $m$ divides a sufficiently high
power of $k$ or equivalently that $m$ is only divisible by primes
dividing $k$.

We have  the  following (see~\S\ref{non-orient-pf} for the proof).
\begin{theorem} 
\label{non-orient-thm}
(i) The stack $\calM_{\varrho,n}$ has rational count for $r=1$
($\varrho<0$) and polynomial count for $r>1$ ($\varrho\geq 0$).

(ii) Let
 \beq
\label{W-fmla}
W_{\varrho,n}(q):=2V_{\varrho,n}(q)+(q-2)V_{2\varrho,n/2}(q)
+\tfrac12(q-1)
\left(V_{\varrho,n/2,2}(q^2)-V_{2\varrho,n/2,2}(q)\right), 
\eeq
 where $V_{\varrho,n}$ and $V_{\varrho,n,k}$ are defined in~\eqref{V_n-defn}
and ~\eqref{V_n-defn-1} for  integer $n$  and set to zero if $n$ is not an
integer.  Then
$$
M_\varrho(q,T)=\Exp\left(\sum_{n\geq 1}W_{\varrho,n}(q) T^n\right)
$$

(iii) The $E$-series of $\calM_{\varrho,n}$ is given by
$$
E(\calM_{\varrho,n};q)=q^{\varrho\binom n2}{\rm
  Coeff}_{T^n}\left(\Exp\left(\sum_{n\geq 1}W_{\varrho,n}(q) T^n\right)\right). 
$$
\end{theorem}


We give  some examples to illustrate the theorem.

\bigskip
\noindent
(i) $\varrho=-1$

\medskip
Here $r=1$ and
$$
|\calM_{1,n}(\F_q)|=\frac{I_n(q)}{|GL_n(\F_q)|},
$$
where $I_n(q)$ denotes the number of involutions in
$\G(\F_q)$, a polynomial in $q$. Concretely,
$$
I_n(q):=|\{x \in \G(\F_q) \,|\, x^2=1\}|
$$
with first few values
\bes
I_1(q)&=& 2 \\
I_2(q)&=& q^2 + q + 2\\
I_3(q)&=& 2q^4 + 2q^3 + 2q^2 + 2\\
I_4(q)&=& q^8 + q^7 + 4q^6 + 3q^5 + 3q^4 + 2q^3 + 2\\
I_5(q)&=& 2q^{12} + 2q^{11} + 4q^{10} + 4q^9 + 6q^8 + 4q^7 + 4q^6 +
2q^5 + 2q^4 + 2\\
\ees
 We have (see Corollary~\ref{rho=-1})
\begin{equation}
\label{r=1}
M_{-1}(q,T)=\sum_{n\geq 0}\frac{q^{\binom n2}I_n(q)}{|\GL_n(\F_q)|}\,  T^n=
\Exp\left(\frac 2{(q - 1)}T + \frac 1{(q + 1)}T^2\right).
\end{equation}
This is a $q$-analogue of the known generating series for the number
of involutions $t_n$ in the symmetric group $S_n$. Namely,
$$
\sum_{n\geq 0}\frac{t_n}{n!}\, T^n=e^{T+\tfrac12T^2}.
$$

\bigskip
\noindent
(ii) $\varrho=0$
\medskip

The first few terms of $M_0(q,T)$ are
$$
M_0(q,T)= 1 + 2T + (q + 3)T^2 + (2q + 6)T^3 + (q^2 + 4q + 9)T^4
+\cdots
$$
 We have
$$
Z_0(q,T)=\prod_{n\geq 1}(1-T^n)^{-1}
$$
and hence $V_{0,n}(q)=1$ for all $n$. It follows from
Theorem~\ref{non-orient-thm} that 
\begin{equation}
\label{M_0-fmla}
\Log\left(M_0(q,T)\right)=2T+qT^2+2T^3+qT^4+\cdots
\end{equation}
or  equivalently,
\begin{equation}
\label{M_0-fmla-1}
M_0(q,T)=\prod_{n\geq 1}(1-T^{2n-1})^{-2}(1-qT^{2n})^{-1}.
\end{equation}

\begin{remark}
  As pointed out by Frobenius and Schur~\cite{Frobenius-Schur} the
  change of variables $x\mapsto xz^{-1}$
  in an arbitrary group takes the equation $x^2z^2=1$
  to $z^{-1}xz=x^{-1}$.
  Hence $M_{0,n}(q)$  equals the number of real conjugacy classes in $\G(\F_q)$
  (i.e. classes of elements conjugate to their inverses). In this form
  \eqref{M_0-fmla-1} was first proved by
  Gow~\cite{Gow},~\cite{Waterhouse}.
\end{remark}

\bigskip
\noindent
(iii) $\varrho>0$

\medskip The expression for $M_{\varrho,n}(q)$ in
Theorem~\ref{non-orient-thm} is easy to program; we give below the
first few values for $r=3, (\varrho=1)$. 


\bes
M_{1,1}(q) &=& 2q - 2\\
M_{1,2}(q) &=& 3q^4 - 2q^3 - 3q^2 + 2\\
M_{1,3}(q) &=& 2q^9 - 2q^8 + 4q^7 - 12q^6 + 10q^5 - 6q^4 + 6q^3 - 2q^2 + 2q - 2\\
\ees

\begin{remark}
\label{irred-comp}
(i) Consider the affine variety  
$$
\calU_r:=\left\{(A_1,\dots,A_r)\in\G^r\,|\, A_1^2\cdots
    A_r^2={\rm I}_n\right\}, \qquad  r\geq1.
$$
By the Weil conjectures the leading coefficient of the polynomial
$|\calU_r(\F_q)|$ equals the number of irreducible components of
largest dimension of $\calU_r$ over the complex numbers.  It is not
difficult to deduce from Theorem~\ref{non-orient-thm} the value of
this leading coefficient. (In general, counting points over $\F_q$
would not yield anything about components of non-maximal dimension.)

On the other hand the number of irreducible components and their
dimensions is computed in~\cite[Thms 2.1, 2.2, 2.3, 3.2 and
Prop. 3.4]{BKCh}.  It is a pleasant exercise to verify that everything
checks out. Here is a sketch on determining the leading coefficients
of the counting polynomials. It is enough to compute the leading
coefficient of $W_{\rho,n}$. We start by verifying that the highest
power of $q$ in $V_{\rho,n}(q)$ is $\rho n(n+1)/2$. Hence the first
term in~\eqref{W-fmla} dominates for $n$ odd and contributes $2$ as
its leading coefficient. The first term also dominates with
coefficient $2$ for $n$ even as long as $\rho n(n+1)/2$ is bigger than
$1+\rho(n/2+1)n/2$. This happens if $\rho>1$ or $\rho=1$ and $n>2$. In
the special case $\rho=1$ and $n=2$ all the terms in~\ref{W-fmla}
actually contribute giving a total leading coefficient of $3$.  If
$\rho=0$ and $n$ is even on the other hand, the remaining terms
in~\eqref{W-fmla} rather than the first term are the ones that
contribute to the highest power of $q$ with coefficient
$1=1+(1/2-1/2)+(1/4-1/4)+\cdots$. The case $\rho=-1$ is easy to check
(see also~\ref{div-by-2}).

Here is a small table of leading coefficients summarizing the
situation.
\begin{center}
\begin{tabular}{c|cccccc}

$r\backslash n$  &$1$&$2$&$3$&$4$&$5$&$\cdots$ \\
\hline

$1$ & $2$&$1$&$2$&$1$&$2$&$\cdots$\\
$2$ & $2$&$1$&$2$&$1$&$2$&$\cdots$\\
$3$ & $2$&$3$&$2$&$2$&$2$&$\cdots$\\
$4$ & $2$&$2$&$2$&$2$&$2$&$\cdots$\\
$5$ & $2$&$2$&$2$&$2$&$2$&$\cdots$\\
$\vdots$&&&$\vdots$&&&
\end{tabular}
\end{center}
Notice the already mentioned special case $r=3,n=2$ where the leading
coefficient is $3$ breaking the pattern; see~\cite[\S 3]{BKCh} for the
explicit description of the irreducible components.

(ii) It seems that $M_{\varrho,n}(q)\equiv 0 \bmod 2$ for all $\varrho>0$
and $n$ odd. Perhaps this is a consequence of the involution
$A\mapsto -A$ in $G$ acting without fixed points on the irreducible
components of $\calU_r$, as it happens for $r=1$ (see
Remark~\ref{div-by-2}), but we have not pursued this further.
\end{remark}

\subsection{$E$-series of $\calM_\calC^\varepsilon$ in case (B)}
\label{case-ii}

We now consider the case (B) where $\epsilon=\sigma$ and to alleviate
the notation we write $\calM_\calC$ instead of
$\calM_\calC^\varepsilon$.  Let $\calP$ be the set of all partitions
and let $r>0$ be an integer. As in~\cite{HLRV} we introduce the Cauchy
function
$$
 \label{cauchygk}
\Omega(z,w)=\Omega_{r,k}(z,w):= \sum_{{\lambda}\in {\cal P}}{\cal
  H}_{r,\lambda}(z,w) \left(\prod_{i=1}^k  
  \tilde{H}_\lambda(\x_i;z^2,w^2)\right),
$$
where 
$$
{\cal H}_{r,\lambda} (z,w):=\prod \frac{(z^{2a+1}-w^{2l+1})^r}
{(z^{2a+2}-w^{2l})(z^{2a}-w^{2l+2})}
$$
is the $(z,w)$-deformed hook function with exponent $r$ and
$\tilde{H}_\lambda(\x_i,z,w)$ denotes the modified Macdonald symmetric
function in the variables $\x_i$. As in~\cite[(2.4.11)]{HRV} it is
easily checked that
\begin{equation}
 \label{euler-spec-hook-pol}
\calH_{r,\lambda}(\sqrt{q},1/\sqrt{q})=q^{-\tfrac12\varrho\langle\lambda,\lambda\rangle}
H_\lambda(q)^\varrho, \qquad
\varrho=r-2. 
\end{equation}

We stress the fact that unlike~\cite{HLRV}, however, the integer $r$
is not necessarily even. In particular, exchanging $z$ and $w$
involves a sign if $r$ is odd. More precisely,
\begin{equation}
{\cal H}_{r,\lambda} (w,z)=(-1)^r
{\cal H}_{r,\lambda'} (z,w),
\end{equation}
where $\lambda'$ is the dual partition.
 
For a $k$-tuple of partitions ${\mu}=( \mu^1,\dots,\mu^k)$ of $n$ let
\begin{equation}
\label{H-defn}
\H_\mu(z,w):= (z^2-1)(1-w^2)\left\langle \Log
  \left(\Omega(z,w)\right),h_\mu\right\rangle,   
\end{equation}
where $h_\mu=h_{\mu^1}(\x_1)\cdots h_{\mu^k}(\x_k)$ and
$h_{\mu ^i}(\x_i)$ is the complete symmetric function in the variables
$\x_i$.

We obtain  the  following (see~\S\ref{orient-pf} for the proof).
\begin{theorem}
\label{orient-thm}
We have
\begin{equation}
\E(\calM_\calC;q)=\frac{q^{d_\mu/2}}{q-1}\H_\mu\left(\sqrt{q},
  \frac{1}{\sqrt{q}}\right),
\label{E-series}\end{equation}
where $\mu$ is defined from $\calC$ as in~\S~\ref{orient-pf} and
$d_\mu$ is as in Theorem \ref{maintheo1}. 
\end{theorem}

\begin{remark}
\label{orientable}
(i) It follows from  Theorem~\ref{orient-thm} that if $r=2h$ is even
and $\Sigma'$ is 
an orientable compact Riemann surface of Euler characteristic $r-2$
and punctures $S'=\{\alpha'_1,\dots,\alpha_k'\}$, then comparison with
\cite{HLRV} shows that the right hand side of~\eqref{E-series} is also
the $E$-series of the quotient stack
$$
\calM_\calC':=\left[\left\{\rho\in\Hom\big(\pi_1(\Sigma'\backslash
    S')\big),\G\,\left|\, \rho(z_i)\in C_i\right\}\right./\G\right]. 
$$
where $z_i$ is a single loop around the puncture $\alpha_i'$. We thus
have
$$
E(\calM_\calC;q)=E(\calM_\calC';q).
$$

(ii) We may construct a correspondence between $\calM_\calC$ and
$\calM_\calC'$ in (i), which gives some indication of how these stacks
are related. We only give a sketch of the construction here for the
case $r=2$.  For fixed $h\in \GL_n$ consider the following varieties:
$$
\calA:=\{(x,z)\,|\, xz\sigma(x)z^{-1}=h\}\subseteq \GL_n\times \GL_n,\qquad
\calB:=\{(x,z)\,|\, xzx^{-1}z^{-1}=h\} \subseteq \GL_n\times \GL_n,
$$
as well as
$$
\calC:=\{(x,z,u)\,|\, \prescript{t}{}\!x=uxu^{-1}, x(zu)z(zu)^{-1}=h\}
\subseteq \GL_n\times \GL_n\times\GL_n.
$$
Note that changing $x$ to $xz^{-1}$  and  then  $z$ to
$\prescript{t}{}\!z$ yields an isomorphism between $\calA$ and the
variety $\{(x,z)\,|\, x\sigma(x)z\sigma(z)=h\}$.

We have natural maps
\[
\begin{tikzcd}
&\arrow[swap]{dl}{p_\calA}  \calC \arrow{dr}{p_\calB} \\
\calA  && \calB
\end{tikzcd}
\]
where $p_\calA(x,z,u):=(x,z)$ and $p_\calB(x,z,u)=(x,zu)$. 
These maps are equivariant for the natural action of
$\GL_n$ on the  three varieties given by
$$
x\mapsto  gxg^{-1},\qquad
z\mapsto gz\prescript{t}{}\!g,\qquad
u\mapsto \prescript{t}{}\!g^{-1}ug^{-1},\qquad g  \in \GL_n.
$$
Since any $x\in \GL_n$ is conjugate to its transpose the maps
$p_\calA,p_\calB$ are surjective with fibers isomorphic to the
stabilizer of $x$ in $\GL_n$.  In particular, for any finite field
$\#\calA(\F_q)=\#\calB(\F_q)$.

(iii) Our surface $\Sigma$ arises from a pair $(\tilde{\Sigma},\sigma)$ with
$\tilde{\Sigma}$ a compact Riemann surface and an
orientation-reversing involution on $\tilde{\Sigma}$ whose fixed-point
set $\tilde{\Sigma}^\sigma$ is empty. In~\cite{BL}, the authors
compute the $E$-polynomials of some character varieties coming from
pairs of the form $(\tilde{\Sigma},\sigma)$ with
$\tilde{\Sigma}^\sigma\neq\emptyset$.
\end{remark}

\begin{remark}
  (i) As in~\cite{HLRV}, there is a natural deformation of the formula
  for $E(\calM_\calC;q)$ to two variables. One might be tempted to
  extend the conjecture on Hodge numbers made there but this may not
  be that straightforward. We do however check one potential case of
  the putative conjecture when $r=k=1$ in~\S\ref{conj-r=k=1} (see
  Theorem~\ref{r=k=1evid}).

  (ii) It was conjectured in~\cite{HLRV}, in particular, that for $r$
  even $\H_\mu(z,w)$ is a polynomial with integer coefficients.  This
  was recently proved by Mellit~\cite{Mellit1}. Note that in general
  $\H_\mu(z,w)$ is not in fact polynomial for $r=1$ (see for example
  Lemma~\ref{r=k=1}). It appears, however, that in general for $r$ odd
  the denominator is fairly small. For example, for $k=1$ we seem to
  just have denominator $q^2+1$ and then only for $n\equiv 2 \bmod 4$
  and when all parts of $\mu$ are divisible by~$2$.  On the other
  hand, note that the specialization
  $q^{d_\mu/2}\H_\mu\left(\sqrt{q}, \frac{1}{\sqrt{q}}\right)$ for any
  $r>1$ {\it is} a polynomial by~\eqref{euler-spec-hook-pol}.
 
  (iii) In any case, we infer a purely combinatorial identity
  involving Macdonald polynomials. Namely, we conjecture (see
  Conjecture~\ref{0conj}) that
$$
(z^2-1)(1-w^2)\Log\left(  \sum_{\lambda\in \P}
\prod \frac{(z^{2a+1}-w^{2l+1})}
{(z^{2a+2}-w^{2l})(z^{2a}-w^{2l+2})}
  \tilde{H}_\lambda(\x;z^2,w^2)\right)
=(z-w)m_{(1)}(\x)+  \frac 1{z^2+1}m_{(2)}(\x) + m_{(1^2)}(\x)
$$
where for a partition $\lambda$, we denote by $m_\lambda(\x)$ the
corresponding monomial symmetric function. 

\end{remark}

\paragraph{Acknowledgements} It is a pleasure to thank T. Scognamiglio for very
useful remarks on this paper and M. Olsson for having read the part on cohomology of stacks. We are grateful for the
hospitality of the institutions were this work was carried out: the first author would like to thank the University of Texas at Austin and the ICTP,  and the
second author would like to thank Oxford University and the
Universit\'e de Caen were the work was started several years ago and
IHP and the Universit\'e de Paris were it was completed.

\section{Preliminaries}

\subsection{Log and Exp}

We let $\Lambda$ be the field $\Q(x_1,\dots,x_r)$ where
$x_1,\dots,x_r$ are indeterminate which commute and $\Lambda[[T]]$ the
ring of series with coefficients in $\Lambda$. 

Consider  $$\psi_n:\Lambda[[T]]\rightarrow\Lambda[[T]],\,
f(q,T)\mapsto f(x_1^n,\dots,x_r^n,T^n).$$The $\psi_n$ are called the
\emph{Adams operations}.  

Define $\Psi:T\Lambda[[T]]\rightarrow T\Lambda[[T]]$
by $$\Psi(f)=\sum_{n\geq 1}\frac{\psi_n(f)}{n}.$$Its inverse is given
by $$\Psi^{-1}(f)=\sum_{n\geq 1}\mu(n)\frac{\psi_n(f)}{n}$$where $\mu$
is the ordinary M\"obius function.  

Define $\Log:1+T\Lambda[[T]]\rightarrow T\Lambda[[T]]$ and its inverse
$\Exp:T\Lambda[[T]]\rightarrow 1+T\Lambda[[T]]$ as  

$$\Log(f)=\Psi^{-1}\left(\log(f)\right)$$and $$\Exp(f)=\exp\left(\Psi(f)\right).$$

They satisfy the following obvious properties.

$$
\Log(f\cdot g)=\Log(f)+\Log(g),\hspace{.5cm}\Exp(h+l)=\Exp(h)\cdot\Exp(l).
$$
They also commute with the Adams operations, namely for any integer $r>0$, we have 

$$
\Log\circ\psi_r=\psi_r\circ\Log,\hspace{.5cm}\Exp\circ\psi_r=\psi_r\circ\Exp.
$$

\begin{remark}
\label{Log-minus}
 Note that the map $T\mapsto -T$ is not preserved under $\Log$ and
 $\Exp$ as $1+q^iT^j=(1-q^{2i}T^{2j})/(1-q^iT^j)$.
\end{remark}

\begin{remark}Let $f\in 1+T\Lambda[[T]]$. If we write
$$
\log\,(f)=\sum_{n\geq 1}\frac{1}{n}U_nT^n,\hspace{1cm}\Log\,(f)=\sum_{n\geq 1}V_nT^n,
$$
then

$$
V_r(q)=\frac{1}{r}\sum_{d | r}\mu(d)\psi_d(U_{r/d}).
$$
\label{LogM}
\end{remark}
We have the following results (details may be found for instance in
Mozgovoy~\cite{mozgovoy}).

For $g\in \Lambda$ and $n\geq 1$ we put 

$$
g_n:=\frac{1}{n}\sum_{d | n}\mu(d)\psi_{\frac{n}{d}}(g).
$$
This is the M\"obius inversion formula of $\psi_n(g)=\sum_{d| n}d\cdot g_d$.

\begin{lemma} Let $g\in \Lambda$ and $f_1,f_2\in 1+T\Lambda[[T]]$ such that

$$
\log\,(f_1)=\sum_{d=1}^\infty g_d\cdot\log\,(\psi_d(f_2)).
$$
Then m
$$
\Log\,(f_1)=g\cdot \Log\,(f_2).
$$
\label{moz}\end{lemma}

\subsection{Mixed Poincar\'e series}\label{MPS}

Let $K$ be $\overline{\F}_q$ and choose a prime $\ell$ which does not
divide $q$.  Let $\frak{X}_o$ be an algebraic stack of finite type
defined over $\mathbb{F}_q$, whose lift to $K$ is denoted by
$\frak{X}$. We denote by $H_c^i(\frak{X},\bar\Q_\ell)$ the compactly
supported $i$-th $\ell$-adic cohomology group of $\frak{X}$ as defined
in~\cite{LO1},~\cite{LO2}.

We denote by $F:\frak{X}\rightarrow\frak{X}$ the geometric Frobenius
and by $F^*$ the induced Frobenius on $\ell$-adic cohomology. Let
$W^k_\bullet$ be the weight filtration on
$H_c^k(\frak{X},\bar\Q_\ell)$, i.e. the $F^*$-stable increasing
filtration such that for all integer $n>0$, the eigenvalues of
$(F^*) ^n$ on the subquotient $W^k_m/W^k_{m-1}$ are pure of weight
$nm$.  

We define then the mixed Poincar\'e series of $\frak{X}$ as 
$$
\rmH_c(\frak{X};x,t):=\sum_{k,m}{\rm dim}\, (W^k_m/W^k_{m-1})\, x^{m/2}t^k.
$$
When it is well defined (i.e. when the sum $\sum_k(-1)^k{\rm dim}\,
(W^k_m/W^k_{m-1})$ is finite) we let the $E$-series of $\frak{X}$
be
$$
\E(\frak{X};x):=\rmH_c(\frak{X};x,-1)=\sum_m\sum_k(-1)^k{\rm dim}\,
(W^k_m/W^k_{m-1})\, x^{m/2}. 
$$

\begin{remark}Let $X/\C$ be a separated scheme of finite type over
  $\C$. The compactly supported  cohomology groups $H_c^i((X/\C)^{\rm
    an},\Q)$ carry a  mixed Hodge structure defined by Deligne and so
  one can define the corresponding $E$-polynomial  
$$
\E(X/\C;x,y)=\sum_{i,j}\sum_k(-1)^kh_c^{i,j;k}(X/\C)x^iy^j,
$$
where $\{h_c^{i,j;k}(X/\C)\}_{i,j}$ are the mixed Hodge numbers of $H_c^k((X/\C)^{\rm an},\C)$. 

We then consider
$$
\E(X/\C;x):=E(X/\C;\sqrt{x},\sqrt{x})=\sum_r\sum_k(-1)^k\sum_{i+j=r}h_c^{i,j;k}(X/\C) x^{r/2}.
$$
If $X/\C$ is projective and smooth then the cohomology is pure of
weight $k$, i.e. $h_c^{i,j;k}$ are zero unless $i+j=k$, and so  
$$
\E(X/\C;x)=\sum_k(-1)^k{\rm dim}\, H_c^k(X/\C,\C)\, x^{k/2}.
$$
Let $R$ be a subring of $\C$ which is finitely generated as a
$\Z$-algebra and let $X/R$ be an $R$-scheme of finite type such that
$X/\C$ is obtained from $X/R$ by scalar extension. Then there is an
open subset $U$ of ${\rm Spec}(R)$ for which the following is true :
for any ring homomorphism $\varphi:R\rightarrow\F_q$ such that the
image of ${\rm Spec}(\F_q)\rightarrow{\rm Spec}(R)$ is in $U$ we have
$$
\E(X/\C;x)=\E(X/_\varphi\overline{\F}_q;x).
$$
If $X/\C$ is smooth projective then this is true because $H_c^k(X/\C,\C)$ and $H_c^k(X/_\varphi\overline{\F}_q,\bar  \Q_\ell)$ are pure of weight $k$~\cite[Th\'eor\`eme I.6] {WeilI}.

The general case reduces to the smooth projective case using the fact
that $E$-polynomials are additive with respect to disjoint unions and
that we always have a decomposition    
$$
[X/\C]=[S/\C]-[T/\C]
$$
 in the Grothendieck group of the category of separated $\C$-schemes of finite type  with $S/\C$ and $T/\C$ both projective and smooth~\cite[Appendix, Lemma 6.1.1]{HRV}.
\end{remark}

\begin{theorem} Let $G$ be a connected linear algebraic group over $K$
  acting on a separated scheme $X$ of finite type over $K$. Assume that $X$, $G$
  and the action are all defined over $\F_q$. The $E$-series of the
  quotient stack $[X/G]$ is well-defined and   
$$
\E([X/G];x)=\E(X;x)\, \E({\rm B}(G);x),
$$
where ${\rm B}(G):=[{\rm Spec}(K)/G]$ is the
classifying stack of $G$.
\label{Behrend}\end{theorem}

\begin{proof}We consider the cartesian diagram

$$
\xymatrix{X\ar[r]\ar[d]&[X/G]\ar[d]\\
{\rm Spec}(K)\ar[r]&{\rm B}(G)}
$$
Following the same lines as in~\cite[\S 2.5]{B} with compactly supported cohomology instead, we show that we have an $E_2$ spectral sequence of finite-dimensional  $\bar  \Q_\ell$-vector spaces

$$
H_c^i({\rm B}(G),\bar  \Q_\ell)\otimes H_c^j(X,\bar  \Q_\ell)\Rightarrow H_c^{i+j}([X/G],\bar  \Q_\ell)
$$
which is compatible with the action of $F^*$. The theorem now follows
from the above spectral sequence as in the proof of the Lefschetz
trace formula in~\cite[\S 2.5]{B}.
\end{proof}

\begin{definition}
\label{ratnl-count-defn}
An algebraic stack of finite type $\frak{X}$ defined over
$\mathbb{F}_q$ has \emph{rational count} if there exists a rational
function $Q(t)\in\Q(t)$ such that for all integer $n>0$, we have

$$
\left|\frak{X}(\F_{q^n})\right|=Q(q^n).
$$
\end{definition}

\begin{remark}
Note that if $X$ and $G$ are as in Theorem~\ref{Behrend} then
$$
\left|[X/G](\F_q)\right|=\frac{|X(\F_q)|}{|G(\F_q)|}
$$
and so $[X/G]$ has rational count if and only if $X$ has polynomial
count. Indeed,  it follows that $X$ has rational count and therefore
$|X(\F_{q^r})|=Q(q^r)$ for $r=1,2,\ldots$ for some $Q(t)\in\Q(t)$
independent of $r$.  Since $|X(\F_{q^r})|$ is an integer,  $Q(t)$  must
be a polynomial in $\Q[t]$. 
\end{remark}

\begin{theorem}
\label{ratnl-count}
Let $X$ and $G$ be as in Theorem \ref{Behrend}. If $[X/G]$ has
rational count with counting rational function $Q(t)$, then
$$
\E([X/G];x)=Q(x).
$$
\end{theorem}

\begin{proof}By Theorem \ref{Behrend} we are reduced to the prove the theorem in the two following cases :

(i) $G=1$.

(ii) $X$ is a point.

\noindent The case (ii) follows from an explicit computation of the dimension of $H^i_c({\rm B}(G),\bar  \Q_\ell)$ which is pure of weight $i$ (see~\cite{HodgeIII}\cite{B}).

Let us prove (i).  We have for all $r>0$
\begin{equation}
|X(\F_{q^r})|=\sum_k(-1)^k{\rm Tr}\left((F^*)^r,H_c^k(X,\bar  \Q_\ell\right).
\label{x}\end{equation}
Then

\begin{equation}
\sum_k(-1)^k\sum_i{\rm Tr}\left((F^*)^r,W^k_i/W^k_{i-1}\right)=\sum_i\left(\sum_k(-1)^k\sum_{j=1}^{s_{k,i}}(\lambda_{k,i;j})^r\right)q^{ir/2},
\label{vb}\end{equation}
where $\lambda_{k,i;1}q^{i/2},\dots,\lambda_{k,i;s_{k,i}}q^{i/2}$ are the eigenvalues (counted with multiplicities) of $F^*$ on $W^k_i/W^k_{i-1}$. 
\bigskip

If $X$ has polynomial count with counting polynomial $P(T)=\sum_ic_i T^i$, then for all $r>0$ we have

$$
\sum_k(-1)^k\sum_{j=1}^{s_{k,i}}(\lambda_{k,i;j})^r=\begin{cases}c_{i/2} &\text{ if } i \text{ is even,}\\0&\text{ otherwise.}\end{cases} 
$$
The result is now a consequence of the following lemma.

\end{proof}
\begin{lemma}Let $k$ be any field and let 
  $a_1,\dots,a_s,\lambda_1,\dots,\lambda_s\in k^\times$ and $d\in k$ be such that for all integers $r>0$ we
  have 
\begin{equation}
\sum_{i=1}^sa_i\lambda_i^r=d.
\label{id}\end{equation}
Then $\sum_{i=1}^s a_i=d$.
\end{lemma}

\begin{proof} From (\ref{id}) we get an identity 

$$
\sum_i\frac{a_i\lambda_i T}{1-\lambda_i T}=\frac{dT}{1-T}
$$
of rational functions in $T$. Letting $T\mapsto\infty$ (or replacing $T$ by $1/t$ and plugging $t=0$) yields the result.
\end{proof}

\subsection{Mass formula}

Consider the group $\Pi$ with presentation
$$
\Pi=\left\langle a_1,\dots,a_r,x_1,\dots,x_k\,\left|\, a_1^2\cdots a_r^2x_1\cdots x_k=1\right\rangle\right..
$$

Let $\G=\GL_n(\overline{\F}_q)$, let $\sigma:\G\rightarrow \G$ be the
involution $g\mapsto {^t}g^{-1}$ and consider the semi-direct product
$\G^+=\G\rtimes\langle \sigma\rangle$.  Consider on $\G$ the
$\F_q$-structure induced by the Frobenius that raises coefficients of
matrices to their $q$-th power and fix a $k$-tuple
$\calC=(C_1,\dots,C_k)$ of conjugacy classes of $\G(\F_q)$. For
$\varepsilon=(\varepsilon_1,\dots,\varepsilon_r)\in \langle
\sigma\rangle^r$ consider the representation variety
$$
\Hom_\calC^\varepsilon(\Pi,\G^+(\F_q)):=\left\{\rho\in
  \Hom(\Pi,\G^+(\F_q))\,\left|\, \epsilon(\rho(a_i))=\varepsilon,\,
    \rho(x_j)\in \iota(C_j)\text{ for all }1\leq i \leq r\text{ and }1\leq
    j\leq k\right\}\right., 
$$
where $\epsilon:\G^+\rightarrow \langle \sigma\rangle$ is the quotient
map and $\iota:G\rightarrow G^+$ the  natural  inclusion.

Recall that for an irreducible complex character $\chi$ of some finite
group $U$, the Schur indicator $c_\chi\in\{-1,0,1\}$ is defined as  
$$
c_\chi:=\frac{1}{|U|}\sum_{u\in U}\chi(u^2).
$$
An irreducible character of $U$ is afforded by a \emph{real} representation
(we call the character \emph{real})
if and only if $c_\chi=1$. We denote by $\widehat{U}$ the set of
irreducible complex characters of $U$ and by $\widehat{U}_{\rm real}$
the subset of real irreducible characters.  It is
know~\cite{Gow} that for $U=\G(\F_q)$ the Schur indicator $c_\chi$
is either~$0$ or $1$. Frobenius and Schur~\cite[(9),
p.197]{Frobenius-Schur} proved the mass formula 
\begin{equation}
\label{Frobenius-Schur-orig}
\frac1{|U|} 
|\{a_1^2\cdots a_r^2=1\}|=\sum_\chi c_\chi\left(\frac{|U|}{\chi(1)}\right)^{r-2},
\end{equation}
where the sum is over all irreducible characters of $U$.

We need the following generalization of~\eqref{Frobenius-Schur-orig}
for our setting. 
\begin{theorem}[Mass formula]
We have
\begin{equation}
\frac{\left|
    \Hom_\calC^\varepsilon(\Pi,\G^+(\F_q))\right|}{|\G(\F_q)|}=  \begin{cases}
\sum_{\chi\in\widehat{G(\F_q)}}\left(\frac{|\G^F|}{\chi(1)}\right)^{r-2}\prod_{i=1}^k\frac{\chi(C_i^F)\,       
    |C_i^F|}{\chi(1)} \text{ if } \varepsilon=\sigma \text{ for all
  }i,\\ 
\sum_{\chi\in\widehat{G(\F_q)}_{\rm real}}\left(\frac{|\G^F|}{\chi(1)}\right)^{r-2}\prod_{i=1}^k\frac{\chi(C_i^F)\, |C_i^F|}{\chi(1)} \text{ if } \varepsilon=1 \text{ for some }i.
\end{cases}
\label{Mass}\end{equation}

\label{Masstheo}
\end{theorem}

\begin{proof}
Let $\calC(\G(\F_q))$ denote the $\C$-vector space of central functions $\G(\F_q)\rightarrow\C$. It is equipped with a convolution product $*$ defined as 
$$
(f_1*f_2)(g)=\sum_{xy=g}f_1(x)f_2(y),
$$
for $f_1,f_2\in\calC(\G(\F_q))$ and $g\in \G(\F_q)$.

Define the function $\eta^\varepsilon:\G(\F_q)\rightarrow\C$ by  
$$
\eta^\varepsilon(y)=\#\{(a_1,\dots,a_r)\in (\G(\F_q))^r\,|\, E_{\varepsilon_1}(a_1)\cdots E_{\varepsilon_r}(a_r)=y\}
$$
where $E_1(a):=a^2$ and $E_\sigma(a)=a\sigma(a)$.Then 
\begin{equation}
\left|\Hom_\calC^\varepsilon(\Pi,\G^+(\F_q))\right|=(\eta^\varepsilon*1_{C_1}*\cdots* 1_{C_k})(1),
\label{form}\end{equation}
where $1_{C_i}$ denotes the function on $\G(\F_q)$ that takes the value $1$ on elements of $C_i$ and $0$ elsewhere.
 
 Denote by ${\rm F}(\widehat{\G(\F_q)})$ be the $\C$-vector space of complex valued functions on $\widehat{\G(\F_q)}$ and let $\calF:{\rm F}(\G(\F_q))\rightarrow {\rm F}(\widehat{\G(\F_q)})$ be defined by 
  $$
 \calF(f)(\chi)=\sum_{g\in\G(\F_q)}\frac{f(g)\chi(g)}{\chi(1)}.
 $$
 It satisfies $\calF(f*g)=\calF(f)\cdot\calF(g)$ where $\cdot$ denotes the pointwise multiplication on ${\rm F}(\widehat{\G(\F_q)})$.
 
 By~\cite[Proposition 3.1.1]{HLRV}, for all $f\in\calC(\G(\F_q))$ we have 
 
 $$
 f(1)=\frac{1}{|\G(\F_q)|}\sum_{\chi\in\widehat{\G(\F_q)}}\chi(1)^2\calF(f)(\chi).
 $$
 We thus deduce from (\ref{form}) that
 
 $$
 \frac{\left|\Hom_\calC^\varepsilon(\Pi,\G^+(\F_q))\right|}{|\G(\F_q)|}=\sum_{\chi\in\widehat{\G(\F_q)}}\frac{\chi(1)^2}{|\G(\F_q)|^2}\, \calF(\eta^\varepsilon)(\chi)\prod_{i=1}^k\frac{\chi(C_i)\,|C_i|}{\chi(1)}.
 $$
 On the other hand we have $\eta^\varepsilon=\eta^{\varepsilon_1}*\cdots*\eta^{\varepsilon_r}$ where
 
 $$
 \eta^{\varepsilon}(y):=|\{a\in \G(\F_q)\,|\, E_{\varepsilon}(a)=y\}|.
 $$
Hence 
 $$
 \frac{\left|\Hom_\calC^\varepsilon(\Pi,\G^+(\F_q))\right|}{|\G(\F_q)|}=\sum_{\chi\in\widehat{\G(\F_q)}}\frac{\chi(1)^2}{|\G(\F_q)|^2}\, \prod_{i=1}^r\calF(\eta^{\varepsilon})(\chi)\prod_{i=1}^k\frac{\chi(C_i)\,|C_i|}{\chi(1)}.
 $$
By ~\cite[Theorem 3]{Gow} we have
 $$
 \eta^\sigma(y)=\sum_{\chi\in\widehat{\G(\F_q)}}\chi(y)
 $$
 for all $y\in \G(\F_q)$ and so
  $$
 \calF(\eta^\sigma)(\chi)=\frac{|\G(\F_q)|}{\chi(1)}.
 $$
 The claim now follows from
  $$
 \calF(\eta^1)(\chi)=\frac{|\G(\F_q)|}{\chi(1)}\, c_\chi.
 $$
\end{proof}

The case $r=k=1,\varepsilon=\sigma$ gives the formula
$$
|\{x\in \GL_n(\F_q)\, |, xx^\sigma=h\}|=\sum_\chi\chi(h)
$$
(see~\cite{FG},~\cite{Gow}).
\section{Character varieties of non-orientable surfaces}

We keep the notation of \S \ref{charactervar} with
$K=\overline{\F}_q$.
We also assume that $S$ is empty (no punctures) so that
$$
\Pi=\left\langle a_1,\dots,a_r\,\left| a_1^2\cdots a_r^2=1\right\rangle\right..
$$
We consider on $\G$, the $\F_q$-structure induced by the Frobenius
that raises coefficients of matrices to their $q$-th power.  

In this section, we prove that the  quotient stack

$$
\calM=\left[\rm{Hom}(\Pi,\G)/\G\right]
$$
has rational count and we compute its $E$-series by counting points
over finite fields thanks to Theorem~\ref{ratnl-count}.
\subsection{Case $r=1$: involutions}
\label{involutions}

For $n\in \N$ let $I_n(q)$ be the number of involutions in
$\G(\F_q)$, i.e.,
$$
I_n(q):=|\{x \in \G(\F_q) \,|\, x^2=1\}|.
$$
Throughout $q$ will be the power of an odd prime unless otherwise
noted. We have, using standard notation
$$
{n\brack r}:=\frac{(q)_n}{(q)_r(q)_{n-r}}
$$
with $(a)_n:=\prod_{k=0}^{n-1}(1-aq^k)$, the following.
\begin{proposition}
\begin{equation}
\label{I-formula}
I_n(q)=\sum_{r=0}^n q^{r(n-r)}\,{n\brack r}
\end{equation}
\end{proposition}
\begin{proof}
If $x\in \G(\F_q)$ is such that $x^2=1$ then it must be conjugate
to a diagonal matrix with, say, $r$ eigenvalues equal to $1$ and $n-r$
eigenvalues equal to $-1$. A calculation shows that the size of this
conjugacy class is $q^{r(n-r)}\,{n\brack r}$.
\end{proof}
In particular, $I_n$ is a polynomial in $q$ with non-negative integer
coefficients (see ~\S\ref{involutions} for the first few values).

\begin{remark}
\label{div-by-2}
i) Note that the map $x\mapsto -x$ induces an involution of the variety
$\calU:=\{x\in\G\,|\,x^2=1\}$ permuting the conjugacy classes of
semisimple elements with eigenvalues in $\{-1,1\}$. When $n$ is odd
this involution does not fix any of these conjugacy classes. This
explains why the polynomials $I_n(q)$, with $n$ odd, are divisible by~$2$.

ii) These conjugacy classes are the irreducible components of the
affine variety $\Hom(\pi_1(\Sigma),GL_n(\C))$,  where $\Sigma$ is
the real projective plane (\cite[Thm 2.1]{BKCh}).
\end{remark}

Note that $\calU=\Hom(\pi_1(\Sigma),\G(\F_q))$. Comparison to the untwisted, orientable
case discussed in~\cite[\S3.8]{HRV} suggests considering the generating series
$$
\sum_{n\geq 0} \frac{q^{n^2/2}I_n(q)}{|\G(\F_q)|}\,T^n.
$$
To avoid dealing with powers of $\sqrt q$ we consider instead
\begin{equation}
\label{I-defn}
I(q,T):=\sum_{n\geq 0} \frac{(-1)^nq^{\binom n2}I_n(q)}{|\G(\F_q)|}\,T^n,
\end{equation}
which amounts to shifting $T$ by a factor of $\sqrt q$; the factor of
$(-1)^n$ simplifies later formulas. Alternatively,
\begin{equation}
I(q,T)=\sum_{n\geq 0}
\frac{I_n(q)}{(q)_n}\,T^n,
\end{equation}
since
\begin{equation}
|\G(\F_q)|=(-1)^nq^{\binom n2}\,(q)_n.
\end{equation}

It follows from the Mass formula~\eqref{Masstheo}  that
\begin{equation}
I_n(q)=\sum_{\chi} \chi(1), 
\end{equation}
where the sum is over the real irreducible characters of
$\G(\F_q)$. Hence we also have
\begin{equation}
\label{I-formula-1}
I(q,T)=\sum_{n\geq 0}\sum_{\chi\,{\rm real}}\,
\frac{\chi(1)}{|\G(\F_q)|}\,(-1)^nq^{\binom n2}T^n.
\end{equation}

\begin{proposition}
\label{I-identity}
The following identity holds
\begin{equation}
\label{Log-I-fmla}
(q-1)\,\Log\left( I(q,T)\right) = -2T+T^2
\end{equation}
or equivalently
\begin{equation}
I(q,T)=\prod_{n\geq 0} \frac{(1-q^nT^2)}{(1-q^nT)^2}.
\end{equation}
\end{proposition}

The identity~\eqref{Log-I-fmla} will follow from a more general formula
that we now describe. Using \eqref{I-formula} we have
$$
I(q,T)=\sum_{n\geq 0} \sum_{r=0}^n \frac{q^{r(n-r)}}{(q)_r(q)_{n-r}}T^n.
$$
This suggests that we introduce another variable and
consider the series
\begin{equation}
I^*(q,X,Y):=\sum_{r\geq 0}\sum_{s\geq 0} \frac{q^{rs}}{(q)_r(q)_s}\,X^rY^s,
\end{equation}
with 
$$
I(q,T)=I^*(q,T,T).
$$
The following generalization of Proposition \ref{I-identity} holds.
\begin{proposition}
\label{I*-identity}
\begin{equation}
(q-1)\,\Log\left( I^*(q,X,Y)\right) = -X-Y+XY
\end{equation}
or equivalently
\begin{equation}
I^*(q,X,Y)=\prod_{n\geq 0} \frac{(1-q^nXY)}{(1-q^nX)(1-q^nY)}.
\end{equation}
\end{proposition}
\begin{proof}
Following Fadeev-Kasahev~\cite{FK}, start with the $q$-binomial theorem
$$
\sum_{n\geq 0} \frac{(X)_n}{(q)_n}\,Y^n=\frac{(XY)_\infty}{(Y)_\infty}
$$
and replace $(X)_n$ by $(X)_\infty/(Xq^n)_\infty$. Now use Euler's
formula
$$
(u)_\infty^{-1}=\sum_{n\geq 0}\frac{u^n}{(q)_n}
$$
with $u=Xq^n$ to finish the proof.
\end{proof}

\begin{remark}
As the characteristic is different from $2$, 
involutions are in bijection with projections
\begin{eqnarray*}
\{x^2=1\} \qquad &\longleftrightarrow & \qquad
\{e^2=e\} \\
x \qquad & \longleftrightarrow & \qquad \tfrac12(1-x).
\end{eqnarray*}
Hence $I_n(q)$ equals the $q$-Stirling number
$S_{n,2}$ that counts the number of
non-trivial splittings of a vector space of dimension $n$ over $\F_q$
into two direct summands~\cite[Example 5.5.2(b),
pp. 45--6]{Stanley},~\cite{Ellerman}.  
\end{remark}

\begin{proposition}
\label{special-cases}
 The following identity holds
\begin{equation}
\label{Z-examples}
\Log\left(Z_{-1}(q,T)\right)=\frac T{(q-1)}+\frac{T^2}{(q^2-1)(q-1)},
\end{equation}
\end{proposition}
\begin{proof} The identity is a specialization of a corresponding
  identity for the   Schur symmetric functions. Indeed from
 ~\cite[p. 45]{Macdonald} we know that 
$$
s_\lambda(1,q,q^2,\ldots)=(-1)^{|\lambda|}\calH_{\lambda'}(q)^{-1}.
$$
On the other hand~\cite[p.76]{Macdonald}
$$
\sum_\lambda
s_\lambda(x_1,x_2,\ldots)=\prod_i(1-x_i)^{-1}\prod_{i<j}(1-x_ix_j)^{-1}
$$
and (Cauchy's formula)
$$
\sum_\lambda s_\lambda(x_1,x_2,\ldots)s_\lambda(y_1,y_2,\ldots)
=\prod_{i,j}(1-x_iy_j)^{-1}.
$$
It follows that
\begin{eqnarray*}
Z_{-1}(q,T)&=&\sum_\lambda\calH_\lambda(q)\,(-T)^{|\lambda|}\\
&=& \prod_{i\geq  0}(1+q^iT)^{-1} \prod_{0\leq i <j}(1-q^{i+j}T^2)^{-1} \\
&=&\prod_{i\geq  0}(1-q^iT)\prod_{i\geq  0}(1-q^{2i}T^2)^{-1}
\prod_{0\leq i <j}(1-q^{i+j}T^2)^{-1} 
\end{eqnarray*}
and hence
$$
\Log\left(Z_{-1}(q,T)\right)=T\sum_{i\geq 0}q^i+T^2\sum_{0\leq i \leq j}
q^{i+j}.
$$
Summing the series finishes the proof.
\end{proof}
We leave to the reader to deduce the following corollary from the
identity~\eqref{Z-examples}.
\begin{corollary}
\label{rho=-1}
We have 
\begin{equation}
\label{Log-M_-1-fmla}
\Log\left(M_{-1}(q,T)\right)=\frac 2{(q - 1)}T + \frac 1{(q + 1)}T^2.
\end{equation}
\end{corollary}
Comparing~\eqref{I-defn} with~\eqref{M-defn} we see that
$M_{-1}(q,T)=I(q,-T)$ hence Corollary~\ref{rho=-1} also follows
from~\eqref{Log-I-fmla} (keeping in mind Remark~\ref{Log-minus}).


\begin{remark}
Proposition~\ref{I*-identity} is essentially the quantum version of
the 5-term relation of the dilogarithm of Fadeev-Kasahev~\cite{FK}. Indeed, if
we let $u,v$ satisfy the relation $vu=quv$ then
$$
I^*(q,u,v)=E(q,v)E(q,u),
$$
where 
$$
E(q,T):=\sum_{n\geq 0} \frac {T^n}{(q)_n}=\prod_{n\geq
  0}(1-q^nT)^{-1}. 
$$
We also have
$$
E(q,T)^{-1}=\sum_{n\geq 0} \frac{(-1)^nq^{n(n-1)/2}}{(q)_n}\,T^n.
$$
It is easily checked by induction that 
$$
v^ru^s=q^{rs}u^sv^r, \qquad \qquad (vu)^n=q^{n(n-1)/2}u^nv^n.
$$
A calculation now shows that if
$$
E(q,X)E(q,Y)E(q,XY)^{-1}= \sum_{m,n\geq 0} c_{m,n}(q) \,X^nY^m
$$
then
$$
E(q,u)E(q,-vu)E(q,v) =\sum_{m,n\geq 0} c_{m,n}(q) \,u^nv^m
$$
and hence Proposition~\ref{I*-identity} is equivalent to
$$
E(q,v)E(q,u)=E(q,u)E(q,-vu)E(q,v).
$$
\end{remark}
\subsection{$\Z\times\Z/2\Z$-orbits on a set}\label{tN}

Fix an infinite  set $X$. Let $F:X\rightarrow X$ be an automorphism of
infinite order such that for all $x\in X$, the set
$\{F^i(x)\,|\,i\in\Z\}$ is finite, and let $\sigma\in{\rm Aut}(X)$ be
an involution that commutes with $F$. Consider the action of
$\Gamma:=\Z\times \Z/2\Z$ on $X$ where the first factor acts via $F$
and the second factor via $\sigma$. 

For $x\in X$, let $r=r(x)$ be the smallest non-negative integer such that 

$$
F^r(x)=\sigma(x),
$$
if one exists, otherwise set $r=\infty$. Let also $d=d(x)$ be the \emph{degree} of $x$, i.e. the size of its $F$-orbit. We will call $\nu=(r,d)$ the $\Gamma$-degree of the $\Gamma$-orbit of $x$; these are of the following three kinds.

\medskip
\begin{center}
\begin{tabular}{c|c|c}
&$\nu$ & $|\nu|$\\
\hline
(i) & $(0,d)$ & $d$ \\
(ii) & $(r,2r)$ & $2r$ \\
(iii)& $(\infty,d)$ & $2d$
\end{tabular}
\end{center}
\medskip
where $|\nu|$ denotes the size of the corresponding $\Gamma$-orbit and $r>0$ in case ii). Orbits of the first kind are of the form $\{x,F(x),\dots,F^{d-1}(x)\}$ with $x\in X^\sigma$. Orbits of the second kind are of the form $\{x,F(x),\dots,F^{2r-1}(x)\}$ with $x$ of degree $2r$ satisfying $F^r(x)=\sigma(x)$. Finally, orbits of the third kind are of the form 

$$
\{x,F(x),\dots,F^{d-1}(x),\sigma(x),\dots,F^{d-1}\sigma(x)\},
$$
where $x$ has degree $d$ and does not satisfy any equation of the form $F^r(x)=\sigma(x)$.

For $\nu$ a given $\Gamma$-degree let $\tilde{N}_\nu(q)$ be the number
of $\Gamma$-orbits of $\Gamma$-degree $\nu$. For integers $r,d> 0$, define
$$
N_d:=|\{x\in X^\sigma\,|\,F^d(x)=x\}|,\hspace{.5cm}N'_r:=|\{x\in X-X^\sigma\,|\,F^r(x)=\sigma(x)\}|,\hspace{.5cm} N_d^\#:=|\{x\in X-X^\sigma\,|\,F^d(x)=x\}|,
$$
We denote by $\mu$ the ordinary M\"obius function.

\begin{proposition} We have 

(i) 

$$
\tilde{N}_{(0,d)}=\frac{1}{d}\sum_{r\mid d}\mu\left(\frac{d}{r}\right) N_d.
$$
Let $\tilde{N}_d^\#$ be the number of $F$-orbit of $X-X^\sigma$ of size $d$. Then

$$
\tilde{N}_d^\#=\frac{1}{d}\sum_{e\mid d}\mu\left(\frac{d}{e}\right)N^\#_e.
$$

(ii) 

$$
\tilde{N}_{(r,2r)}=\frac{1}{2r}\sum_{s\mid r,\; r/s \text{ \rm odd}}\mu\left(\frac r
  s \right) N'_s.
$$

(iii) 

$$
\tilde{N}_{(\infty,d)}=\frac{1}{2}\tilde{N}_d^\#-\begin{cases}
\frac 1 2\tilde N_{(d/2,d)} &\quad \text{ \rm if $d$ is even}\\
0 & \quad \text{ \rm otherwise.}
\end{cases}
$$

\label{orbits}\end{proposition}

\begin{proof} We only prove (ii). Put 
$$
X'_r:=\{x\in X-X^\sigma\,|\, F^r(x)=\sigma(x)\},
$$
and let $X_{(s,2s)}$ be the subset of elements $x$ of $X'_s$ such that $r(x)=s$. Since $\sigma$ is an involution we have:
$$
X'_r=\bigcup_{s\mid r,\,r/s\,{\rm odd}}X_{(s,2s)}.
$$

Hence 
\begin{align*}
N'_r&=\sum_{s\mid r,\,r/s\,{\rm odd}}|X_{(s,2s)}|
\end{align*}
From the M\"obius inversion formula we find that 
$$
|X_{(r,2r)}|=\sum_{s\mid r,\,r/s\,{\rm odd}}\mu\left(\frac{r}{s}\right)N'_s.
$$
We thus deduce ii) by noticing that $\tilde{N}_{(r,2r)}=\frac{1}{2r}\,|X_{(r,2r)}|$.
\end{proof}

\subsection{Colorings on varieties, infinite products formulas}

The first part of this section is a minor extension of~\cite{RV} which
we recall for the convenience of the reader. 

We keep the notation of \S \ref{tN} but here we assume that $X$ is an
algebraic variety over $\overline{\F}_q$ which is defined over $\F_q$
and that $F:X(\oF_q)\rightarrow X(\oF_q)$ is the corresponding
Frobenius endomorphism (for any integer $r\geq 1$, we have $ 
X^{F^r}=X(\F_{q^r})$). Here we use the notation
$\tilde{N}_{(0,d)}(q)$, $\tilde{N}_{(d,2d)}(q)$ and
$\tilde{N}_{(\infty,d)}(q)$  instead of $\tilde{N}_{(0,d)}$,
$\tilde{N}_{(d,2d)}$ and $\tilde{N}_{(\infty,d)}$.  

We also make the assumption that there exists polynomials
$\tilde{N}_{(0,1)}(T), \tilde{N}_{(1,2)}(T),
\tilde{N}_{(\infty,1)}(T)\in \Q[T]$ such that for any finite field
extension $\F_{q^d}$ of $\F_q$, we have  
\begin{align*}
&\tilde{N}_{(0,1)}(q^d)=\tilde{N}_{(0,d)}(q)\\
&\tilde{N}_{(1,2)}(q^d)=\tilde{N}_{(d,2d)}(q)\\
&\tilde{N}_{(\infty,1)}(q^d)=\tilde{N}_{(\infty,d)}(q)\\
\end{align*}
Then there exist also polynomials $N_1(T), N'_1(T), N_1^\#(T)\in
\Q[T]$, such that for any finite field extension $\F_{q^d}$ we have  
$$
N_1(q^d)=N_d,\hspace{.5cm} N'_1(q^d)=N'_d, \hspace{.5cm}N_1^\#(q^d)=N_d^\#.
$$

Denote by $\calP$ the set of all partitions and denote by $0$ the
unique partition of $0$. For $\lambda\in\calP$, we denote by
$|\lambda|$ the size of $\lambda$. Assume given a \emph{weight
  function} $W_\lambda:\calP\rightarrow\Q(T)$, $\lambda\mapsto
W_\lambda(T)$ with $W_0(T)=1$. Let $X/\Gamma$ denotes the set of
$\Gamma$-orbits of $X(\oF_q)$. For a map $f:X/\Gamma\rightarrow \calP$
we let $|f|:=\sum_{\gamma\in
  X/\Gamma}|\gamma|\cdot|f(\gamma)|\in\N\cup\{\infty\}$ be the size of
the support of $f$. We also denote by $O_0,O_1,O_\infty$ the union of
the orbits of the first kind, the second kind and the third kind
respectively (so that $O_0\cup O_1\cup O_\infty=X/\Gamma$). 

Then for $f:X/\Gamma\rightarrow \calP$ with finite support, we put 

$$
W_f(q):=\prod_{\gamma\in
  O_0}W_{f(\gamma)}\left(q^{|\gamma|}\right)\prod_{\gamma\in
  O_1}W_{f(\gamma)}\left(q^{|\gamma|}\right)\prod_{\gamma\in
  0_\infty}W_{f(\gamma)}\left(q^{|\gamma|/2}\right)^2. 
$$
Consider
 $$
 Z(q,T):=\sum_\lambda W_\lambda(q)T^{|\lambda|},\hspace{.5cm} Z_2(q,T):=\sum_\lambda W_\lambda(q)^2T^{|\lambda|}.
 $$

\begin{proposition} We have 
\begin{equation}
\sum_{\{f:X/\Gamma\rightarrow\calP,\,|f|<\infty\}}W_f(q)T^{|f|}=\prod_{d\geq 1}Z(q^d,T^d)^{\tilde{N}_{(0,d)}(q)}\prod_{r\geq 1}Z(q^{2r},T^{2r})^{\tilde{N}_{(r,2r)}(q)}\prod_{d\geq 1}Z_2(q^d,T^{2d})^{\tilde{N}_{(\infty,d)}(q)}.
\label{W}\end{equation}
\end{proposition}

\begin{proof}We have 
$\sum_{\{f:X/\Gamma\rightarrow\calP,\,|f|<\infty\}}W_f(q)T^{|f|}$
\begin{align*}
&=\left(\sum_{f:O_0\rightarrow\calP}\left(\prod_{\gamma\in O_0}W_{f(\gamma)}\left(q^{|\gamma|}\right)\right)T^{|f|}\right)\left(\sum_{f:O_1\rightarrow\calP}\left(\prod_{\gamma\in O_1}W_{f(\gamma)}\left(q^{|\gamma|}\right)\right)T^{|f|}\right)\left(\sum_{f:O_\infty\rightarrow\calP}\left(\prod_{\gamma\in O_\infty}W_{f(\gamma)}\left(q^{|\gamma|/2}\right)^2\right)T^{|f|}\right)\\
&=\prod_{\gamma\in O_0}\left(\sum_\lambda W_\lambda\left(q^{|\gamma|}\right)T^{|\gamma|\cdot|\lambda|}\right)\cdot\prod_{\gamma\in O_1}\left(\sum_\lambda W_\lambda\left(q^{|\gamma|}\right)T^{|\gamma|\cdot|\lambda|}\right)\cdot\prod_{\gamma\in O_\infty}\left(\sum_\lambda W_\lambda\left(q^{|\gamma|/2}\right)^2T^{|\gamma|\cdot|\lambda|}\right)\\
\end{align*}
 \end{proof}

In order to express the left hand side of~\ref{W} as an infinite
product in the variables $q$ and $T$, we need to compute the Log of
the right hand side of~\eqref{W}  (see~\cite{RV} for more details). It
is convenient to work formally and consider the following general case. 

Put
$$
F_0:=\prod_{d\geq 1}\left(\Omega_0(q^d,T^d)\right)^{\tilde{N}_{(0,d)}(q)}, \hspace{.5cm} F_1:=\prod_{r\geq 1}\left(\Omega_1(q^{2r},T^{2r})\right)^{\tilde{N}_{(r,2r)}(q)}, \hspace{.5cm} F_\infty:=\prod_{d\geq 1}\left(\Omega_\infty(q^d,T^{2d})\right)^{\tilde{N}_{(\infty,d)}(q)}
$$
for some $\Omega_i(q,T)\in1+T\Lambda[[T]]$ with $i\in\{0,1,\infty\}$.

For $i=1,\infty$, define $\{H_{i,n}(q)\}_n$ by 
$$
\sum_{n\geq 1}H_{i,n}(q)T^n=\Log\,\left(\Omega_i(q,T)\right).
$$

\begin{theorem}We have
(i)
$$
\Log\,(F_0)=N_1(q)\,\Log\,\left(\Omega_0(q,T)\right).
$$
(ii)
$$\Log\,(F_1)=\frac 1 2N'_1(q)\sum_{m\geq 1}\left(\sum_{j=0}^{v_2(m)}\frac{1}{2^j}H_{1,m/2^j}(q^{2^{j+1}})\right)T^{2m}.
$$
(iii)
$$\Log\,(F_\infty)=\frac{1}{2}N_1^\#(q)\,\sum_{m\geq 1}H_{\infty,m}(q) T^{2m}-\frac 1 2 N'_1(q)\sum_{m\geq 1}\sum_{j=1}^{v_2(m)}\frac 1 {2^j}H_{\infty,m/2^j}(q^{2^j})T^{2m}.
$$
where $v_2$ denotes the valuation at $2$ and $H_{1,x}=H_{\infty,x}=0$ if $x$ is not an integer.

\label{maintheo}\end{theorem}
\begin{proof} The first identity follows from Lemma \ref{moz} and
  Proposition \ref{orbits}(i). For the second identity, we compute
  $\Log(F_1)$ using the two steps procedure mentioned in Remark
  \ref{LogM}. We thus define $\{R_n(q)\}_{n\geq 1}$ by  
$$
\log\,(\Omega_1(q,T))=\sum_{n\geq 1}R_n(q)\frac{T^n}{n}.
$$

We have 
\begin{align*}
\log\,(F_1)&=\sum_{r\geq 1}\tilde{N}_{(r,2r)}(q)\log\,(\Omega_1(q^{2r},T^{2r}))\\
&=\sum_{r\geq 1}\tilde{N}_{(r,2r)}(q)\sum_{n\geq 1}R_n(q^{2r})\frac{T^{2rn}}{n}\\
&=\sum_{m\geq 1}\left(\sum_{r|m}2r\tilde{N}_{(r,2r)}(q)R_{\frac{m}{r}}(q^{2r})\right)\frac{T^{2m}}{2m}\\
&=\sum_{N\geq 1}C_N(q)\frac{T^N}{N}
\end{align*}
where
$$
C_N(q):=\begin{cases}\sum_{r|\frac{N}{2}}2r\tilde{N}_{(r,2r)}(q)R_{\frac{N}{2r}}(q^{2r})&\text{ if }N\text{ is even}\\0& \text{otherwise}\end{cases}.
$$
Then
$$
\Log\, (F_1)=\sum_{n\geq 1}V_{n}(q)T^{n}
$$ 
where $V_n(q):=\frac{1}{n}\sum_{d|n}\mu(d)C_{n/d}(q^d)$.

Hence $V_n(q)=0$ if $n$ is odd and
\begin{align*}
V_{2m}(q)&=\frac{1}{2m}\sum_{\{d|2m\,,\,\frac{2m}{d}\,{\rm even}\}}\mu(d)C_{2m/d}(q^d)\\
&=\frac{1}{2m}\sum_{d|m}\mu(d)C_{2m/d}(q^d)\\
&=\frac{1}{2m}\sum_{d|m}\mu(d)\sum_{r|\frac{m}{d}}2r\tilde{N}_{(r,2r)}(q^d)R_{\frac{m}{rd}}(q^{2rd})\\
&=\frac{1}{m}\sum_{k|m}\sum_{d|k}\mu(d)\frac{k}{d}\tilde{N}_{\left(\frac{k}{d},2\frac{k}{d}\right)}(q^d)R_{\frac{m}{k}}(q^{2k})\\
&=\frac{1}{m}\sum_{k|m}R_{\frac{m}{k}}(q^{2k})\sum_{d|k}\mu(d)\frac{k}{d}\tilde{N}_{\left(\frac{k}{d},2\frac{k}{d}\right)}(q^d)\\
&=\frac{1}{m}\sum_{k|m}R_{\frac{m}{k}}(q^{2k})\sum_{d|k}\mu\left(\frac k d\right)d\tilde{N}_{\left(d,2d\right)}(q^{k/d})
\end{align*}

By Proposition \ref{orbits} we have 
\begin{align*}
r\tilde{N}_{(r,2r)}(q)&=\frac{1}{2}\sum_{s|r,\,r/s\,{\rm odd}}\mu\left(\frac{r}{s}\right)N'_1(q^s)\\
&=\frac{1}{2}\sum_{s|r'}\mu(s)N'_1(q^{r/s}),
\end{align*}
where $r':=r/2^{v_2(r)}$. 

Put $g_s(q):=\mu(s)N'_1(q^{1/s})$ and $f_r(q):=\sum_{s|r}g_s(q)$. Note
that $r\tilde{N}_{(r,2r)}(q)=\frac{1}{2}f_{r'}(q^r)$. We thus have
\begin{align*}
V_{2m}(q)&=\frac{1}{2m}\sum_{k|m}R_{\frac{m}{k}}(q^{2k})\sum_{d|k}\mu\left(\frac k d\right)f_{d'}(q^k)\\
&=\frac{1}{2m}\sum_{k|m}R_{\frac{m}{k}}(q^{2k})\sum_{d|k'}\sum_{j=0}^{v_2(k)}\mu\left(2^{v_2(k)-j}\frac{k'}{d}\right)f_d(q^k)\\
&=\frac{1}{2m}\sum_{k|m}R_{\frac{m}{k}}(q^{2k})\sum_{d|k'}f_d(q^k)\left(\sum_{j=0}^{v_2(k)}\mu\left(2^{v_2(k)-j}\frac{k'}{d}\right)\right)\\
\end{align*}

Since in the above sum $k'/d$ is odd, the sum on the right hand side
equals $0$ unless $v_2(k)=0$.  Hence  
\begin{align*}
V_{2m}(q)&=\frac{1}{2m}\sum_{\{k|m,\,k \text{ odd}\}}R_{\frac{m}{k}}(q^{2k})\sum_{d|k}\mu\left(\frac k d\right)f_d(q^k)\\
\end{align*}
By the M\"obius inversion formula we have 
$$
\sum_{d|k}\mu\left(\frac k d\right)f_d(q)=g_k(q).
$$
Hence 
\begin{align*}
V_{2m}(q)&=\frac{1}{2m}N'_1(q)\sum_{\{k|m,\,k \text{ odd}\}}\mu(k)R_{\frac{m}{k}}(q^{2k})\\
&=\frac{1}{2m}N'_1(q)\sum_{k|m'}\mu(k)R_{\frac{m}{k}}(q^{2k})\\
&=\frac 1 2N'_1(q)\left(\sum_{j=0}^{v_2(m)}\frac{1}{2^j}H_{1,m/2^j}(q^{2^{j+1}})\right),
\end{align*}
from which we deduce the second identity.

By Proposition \ref{orbits} (iii) we have
$$
F_\infty=\left(\prod_{d\geq 1}\Omega_\infty(q^d,T^{2d})^{\frac 1 2\tilde{N}_d^\#(q) }\right)\cdot\left(\prod_{n\geq 1}\Omega_\infty(q^{2n},T^{4n})^{-\frac 1 2 \tilde{N}_{(n,2n)}(q)}\right).
$$
where $\tilde{N}_d^\#(q)$ is the number of $F$-orbits in $X(\oF_q)-X^\sigma(\oF_q)$ of size $d$.
Hence

$$
\Log\,(F_\infty)=\frac 1 2 N_1^\#(q)\,\Log\,\left(\Omega_\infty(q,T^2)\right)+\Log\,\left(\prod_{n\geq 1}\Omega_\infty(q^{2n},T^{4n})^{-\frac 1 2 \tilde{N}_{(n,2n)}(q)}\right).
$$
The  result is then a consequence of the fact that 
$$\Log\, \Omega_\infty (q,T^2)=\sum_{m\geq 1}H_{\infty,m}(q) \, T^{2m}.$$
\end{proof}

\subsection{Proof of Theorem~\ref{non-orient-thm}}
\label{non-orient-pf}

By Theorem \ref{Mass} we have
\begin{equation}
\label{Frobenius-Schur}
  \frac{\left|\left\{(x_1,\ldots,x_r)\in \G(\F_q)^r \,|\, x_1^2\cdots
    x_r^2=1\right\}\right|}{|\G(\F_q)|}=\sum_{\chi\in\widehat{\G(\F_q)}_{\rm real}}  
  \left(\frac{|\G(\F_q)|}{\chi(1)}\right)^{r-2}.
\end{equation}

To describe the real characters of $\G(\F_q)$ we will use the idea of
colorings on varieties~\cite{RV}. Let $\sigma \in \Aut(\Gm)$ be the
involution $\sigma(x):=x^{-1}$ We have an action of
$\Gamma:=\Z \times \Z/2\Z$ on $\Gm(\oF_q)$, where the first factor
acts via the Frobenius $F$ and the second via $\sigma$. These orbits
are described in \S \ref{tN}. However note that here we have
$\Gm^\sigma=\{1,-1\}$, and so in the case i) we have only two orbits
of the kind $(0,1)$ which are $\{1\}$ and $\{-1\}$.  For $\nu$ a given
$\Gamma$-degree let $\tilde N_\nu(q)$ be the number of $\Gamma$-orbits
of $\Gamma$-degree $\nu$.  The following formulae for these quantities
hold (see Proposition \ref{orbits}).
\begin{proposition}
(i) 
$$
\tilde N_{(0,1)}=2
$$

(ii)
$$
\tilde N_{(r,2r)}=\frac1{2r}\sum_{s\mid r,\; r/s \text{ \rm odd}}\mu\left(\frac r
  s \right) (q^s-1)
$$

(iii) 
$$
\tilde N_{(\infty,d)} = \frac 1{2} \tilde{N}_d^\#(q) - 
\begin{cases}
\frac 1 2\tilde N_{(d/2,d)} &\quad \text{ \rm if $d$ is even}\\
0 & \quad \text{ \rm otherwise.}
\end{cases}
$$
with 

$$
\tilde{N}_d^\#(q)=\frac 1{d} \sum_{e\mid d} \mu\left(\frac d
    e\right)(q^e-3).
    $$
\end{proposition}

A real character of $\G(\F_q)$ is uniquely described by a map
$\Lambda$ from $\Gamma$-orbits (technically on the dual of $\Gm(\oF_q)$ but for convenience we will still think of $\Gm(\oF_q)$) to
the set of partitions such that
$$
|\Lambda|:=\sum_{\gamma\in \Gm/\Gamma}|\gamma|\cdot|\Lambda(\gamma)|=n.
$$

Recall that for an integer $\varrho$ we defined  (see~\eqref{Z-defn})
$$
Z_\varrho(q,T)=\sum_\lambda \calH_\lambda(q)^\varrho \, T^{|\lambda|},
$$
where $\calH_\lambda(q)$ is the hook polynomial~\eqref{hook-polyn}, and
$\{V_{\varrho,n}(q)\}_n$ by the formula (see~\eqref{V_n-defn})
$$
\sum_{n\geq 1}V_{\varrho,n}(q) T^n:=\Log\,\left(Z_\varrho(q,T)\right).
$$
Note that $V_{\varrho,n}(q)$  is a rational function of $q$ for
$\varrho<0$ and a Laurent polynomial for $\varrho\geq 0$.

 By the Mass Formula \eqref{Frobenius-Schur} we have
\begin{equation}
\label{M_n-defn}
  M_{\varrho,n}(q):= \frac{\left|\left.\left\{(x_1,\ldots,x_r)\in
    \G(\F_q)^r\,\right|\, 
    x_1^2\cdots     x_r^2=1\right\}\right|}{|\G(\F_q)|}=
  \sum_{\{\Lambda:\Gm/\Gamma\rightarrow\calP,\,|\Lambda|=n\}} 
  \left(\frac{|\G(\F_q)|}{\chi_\Lambda(1)}\right)^\varrho, 
\end{equation}
with $\varrho=r-2$.  Using the formula for $\chi(1)$ in terms of hook
polynomials (see for instance~\cite[IV, 6.7]{macdonald}) we obtain the
following.
$$
q^{-\varrho \binom n 2}M_{\varrho,n}(q)=
\sum_{\{\Lambda:\Gm/\Gamma\rightarrow\calP,\,|\Lambda|=n\}}  
\, \calH_\Lambda(q)^\varrho, 
$$

By Formula~\eqref{W} we obtain the following
\begin{proposition}
\begin{equation}
\label{M-fmla}
M_\varrho(q,T)=  1+\sum_{n\geq 1}q^{-\varrho \binom n2} M_{\varrho,n}(q)\,T^n=
\prod_{d\geq 1}Z_\varrho(q^d,T^d)^{\tilde N_{(0,1)}(q)}
\prod_{r\geq  1}Z_\varrho(q^{2r},T^{2r})^{\tilde N_{(r,2r)}(q)}
\prod_{d\geq  1}Z_{2\varrho}(q^d,T^{2d})^{\tilde N_{(\infty,d)}(q)}.
\end{equation}
\end{proposition}
Now we  can complete the proof of Theorem~\ref{non-orient-thm}: (ii)
follows combining~\eqref{M-fmla}  with Theorem~\ref{maintheo}; (i) is
a direct consequence of (ii) given the above observation on the nature
of $V_{\varrho,n}(q)$; (iii) is just a restatement of (ii).

\begin{remark}
The statement (Theorem~\ref{non-orient-thm} (i)) that
$M_{\varrho,n}(q)$ is a polynomial in $q$ for $r>1$ also follows  from
the main result of~\cite{GRV} as the  abelianization of the
fundamental group 
$$
\Pi=\left\langle a_1,\dots,a_r\,\left| a_1^2\cdots a_r^2=1\right\rangle\right..
$$
is infinite for $r>1$ (but finite for $r=1$).
\end{remark}

\section{Orientation cover of non-orientable surfaces}

\subsection{Character varieties}\label{charactervar}

Let $r\geq 1$ be an integer, put $\varrho=r-2$ and denote by $\Sigma$
a compact non-orientable surface of Euler characteristic $-\varrho$
(the connected sum of $r$ real projective planes). Consider a set
$S=\{\alpha_1,\dots,\alpha_k\}$ of $k$-points of $\Sigma$ ($k\geq
1$). Fix a base point $b\in \Sigma\backslash S$. The fundamental group
$\Pi=\pi_1(\Sigma\backslash S,b)$ has the presentation 
$$
\Pi=\left\langle a_1,\dots,a_r,x_1,\dots,x_k\, \left|\, a_1^2\cdots a_r^2x_1\cdots x_k=1\right\rangle\right..
$$
Let $K$ be an algebraically closed field and consider a \emph{generic} $k$-tuple  $\mathcal{C}=(C_1,\dots,C_k)$ of semisimple conjugacy classes of $\G=\GL_n(K)$ such that

$$
\prod_{i=1}^k\det(C_i)=1.
$$
 
 The genericity condition  means that for any  $1\leq s<n$, if we select $r$ eigenvalues of $C_i$ for each $i=1,\dots,k$ (possibly with multiplicities), then the product of the $rk$ selected eigenvalues is not equal to $1$.

 We denote by $\sigma$ the automorphism $\G\rightarrow\G$,
 $x\mapsto {^t}x^{-1}$ and $\G^+$ the semi-direct product
 $\G\rtimes \langle \sigma\rangle$. We consider the representation
 variety
\begin{align*}
\calU_\calC&:=\left\{\rho\in\Hom(\Pi,\G^+)\,\left|\,\forall
             j=1,\dots,r\, \forall i=1,\dots,k,\,  \rho(a_j)\in
             \G\sigma, \rho(x_i)\in \iota(C_i) \right\}\right.\\ 
&=\left\{(A_1,\dots,A_r,X_1,\dots,X_k)\in \G^r\times
  C_1\times\cdots\times C_k\,\left|\, A_1\sigma(A_1)\cdots
  A_r\sigma(A_r)X_1\cdots X_k=1\right\}\right..
\end{align*}

Let $\G$ act on $\calU_\calC$ by
$$
A_j\mapsto gA_j\sigma(g^{-1}),\qquad x_i\mapsto gx_ig^{-1}, \qquad  g
\in G
$$
and we consider the quotient stack
$$
\calM_\calC=[\calU_\calC/\G].
$$

\begin{remark} Note that unlike in the situation of~\cite{HLRV}, even
  though  the $k$-tuple $\calC$ is assumed to be generic, the action
  of $\G$ is far from being free. For instance, if $r=1$, $n=2$ and $k=1$
  with the central matrix $X_1=\diag (-1,-1)$ at the puncture, then
  $\diag(i,-i)$, with $i^2=-1$, is a solution to the equation 
$$
A\sigma(A)=X_1
$$
while the stabilizer of the pair $\big((1,-1),(-1,-1)\big)$ is the
torus of diagonal matrices. 
\end{remark}

\begin{remark}
The character stack $\calM_\calC$ has an alternative description that can
be found in ~\cite{cheng}\cite{FS} and we now describe briefly. Denote by
$p:\tilde{\Sigma}\rightarrow\Sigma$ the orientation covering of
$\Sigma$. This is an unramified double covering with $\tilde{\Sigma}$
a compact Riemann surface of genus $g=r-1$. Considering the
homomorphism ${\rm Gal}(p)\rightarrow {\rm Aut}(\G)$ that maps the
non-trivial element of the Galois group of $p$ to $\sigma$; let us
identify ${\rm Gal}(p)$ with $\langle \sigma\rangle$.  The orientation
character $\chi:\Pi\rightarrow \langle \sigma\rangle$ maps a loop that
preserves the orientation to $1$ and a loop that reverses the
orientation to $\sigma$. It thus maps $a_j$ to $\sigma$ and $x_i$ to
$1$. Denote by $\tilde{\Pi}$ the fundamental group of
$\tilde{\Sigma}\backslash p^{-1}(S)$ with some base point $\tilde{b}$
above $b$.

A representation $\rho\in \calU_{\overline{\calC}}$ defines (by restriction) a representation $\tilde{\rho}:\tilde{\Pi}\rightarrow\G$ making the following diagram commutative
\begin{equation}
\xymatrix{1\ar[r] 
&\tilde{\Pi}\ar[r]^{p_*}\ar[d]^{\tilde{\rho}}
&\Pi\ar[r]^-\chi\ar[d]^{\rho}&\langle
  \sigma\rangle\ar[r]\ar[d]^{{\rm Id}}&1\\ 
1\ar[r]&\G\ar[r]^\iota&\G^+\ar[r]^\pi
&\langle \sigma\rangle\ar[r]&1} 
\label{diag}\end{equation}
In order to understand local monodromies we need to choose the
generators of $\Pi$ precisely. For $i=1,\dots,k$, we let $D_i$ be a
small open neighbourhood (homeomorphic to an open disc in $\C$) of
$\alpha_i$ in $\Sigma$,  such that $p$ is trivial over $D_i$. Let
$\beta_i$ be a point in  $D_i$   and let $\lambda_i$ be a path from
$x$ to $\beta_i$. We choose a single loop $\ell_i$ in $D_i$  based at
$\beta_i$ around $\alpha_i$ and we take the generator $x_i$ of $\Pi$
to be $\lambda_i^{-1}\ell_i\lambda_i$.  
The path $\lambda_i$ lifts to a unique path $\tilde{\lambda}_i$ from
$\tilde{b}$ to some point $\tilde{\beta}_{i1}$ above $\beta_i$. We let
$\tilde{D}_{i1}$ be the connected component of $p^{-1}(D_i)$
containing $\tilde{\beta}_{i1}$, $\tilde{\alpha}_{i1}$ the unique
point above $\alpha_i$ in $\tilde{D}_{i1}$ and $\tilde{\ell}_{i1}$ the
loop in $\tilde{D}_{i1}$ based at $\tilde{\beta}_{i1}$ around
$\tilde{\alpha}_{i1}$ which lifts $\ell_i$. Then we let $x_{i1}$ be
the generator $\tilde{\lambda}_i^{-1}\tilde{\ell}_i\tilde{\lambda}_i$
of $\tilde{\Pi}$. Let $\gamma_\sigma\in \Pi$ be a loop reversing the
orientation. If instead of the path $\lambda_i$ from $x$ to $\beta_i$
we choose the path $\lambda_i\gamma_\sigma$, then we get into the
other connected component of $p^{-1}(D_i)$ with punctures
$\alpha_{i2}$ and we obtain an other generator $x_{i2}$ of
$\tilde{\Pi}$ which is mapped to $\gamma_\sigma^{-1} x_i\gamma_\sigma$
via $p$.  

Therefore if $\rho(x_i)\in \iota(C_i)$ then
$\tilde{\rho}(x_{i1})\in C_i$ and
$\tilde{\rho}(x_{i2})\in\sigma(C_i)$.  Put $C_{ij}=C_i$ if $j=1$ and
$\sigma(C_i)$ if $j=2$, and consider the representation variety
$$
\tilde{\calU}_\calC:=\left\{\tilde{\rho}\in \Hom(\tilde{\Pi},\G)\,\left|\, \tilde{\rho}(x_{ij})\in \iota(C_{ij})\right\}\right.
$$
Say that a pair $(\tilde{\rho}, h_\sigma)\in \tilde{\calU}_\calC\times G$ is $\sigma$-invariant if  we have

$$
\begin{cases}h_\sigma \tilde{\rho}(z)h_\sigma^{-1}=\sigma\left(\tilde{\rho}(\gamma_\sigma^{-1}z\gamma_\sigma)\right)\hspace{.5cm}\text{ for all }z\in\tilde{\Pi}\\
\tilde{\rho}(\gamma_\sigma^2)=h_\sigma^{-1}\sigma(h_\sigma^{-1})\end{cases}.
$$
If the two above conditions are satisfied then the representation $\tilde{\rho}$ can be extended to an homomorphism $\rho:\Pi\rightarrow \G^+$ making the diagram (\ref{diag}) commutative. The group $\G$ acts on the space $\tilde{\calU}_{\calC,\sigma}$ of $\sigma$-invariant pairs as follows

$$
g\cdot (\tilde{\rho},h_\sigma)=(g\cdot\tilde{\rho}, \sigma(g)h_\sigma g^{-1}),
$$
where $g\cdot\tilde{\rho}$ is the representation obtained by composing $\tilde{\rho}$ with  the conjugation by $g$.  

Then~\cite[Theorem 2.2.1]{cheng} we have an isomorphism

$$
\calM_\calC\simeq [\tilde{\calU}_{\calC,\sigma}/\G].
$$
The stack $[\tilde{\calU}_{\calC,\sigma}/\G]$ can be described in
terms of  local systems $\mathcal{L}$ on $\tilde{\Sigma}\backslash
p^{-1}(S)$ satisfying $\mathcal{L}\simeq \sigma^*(\mathcal{L}^\vee)$
(where $\mathcal{L}^\vee$ is the dual local system of $\mathcal{L}$)
with local monodromy in $C_{ij}$ at the puncture $\alpha_{ij}$, see
\cite[\S 6.1.2]{cheng} for more details. 
\end{remark}

\subsubsection{The Cauchy function} For any $i=1,\dots,k$, let $\x_i$
be an infinite set of commuting variables and let $\calP$ denotes the
set of partitions.  

As in~\cite{HLRV}, for an integer $r>0$ we introduce the Cauchy function
$$
 \label{cauchygk}
\Omega(z,w)=\Omega_{r,k}(z,w):= \sum_{{\lambda}\in {\cal P}}{\cal
  H}_{r,\lambda}(z,w) \left(\prod_{i=1}^k  
  \tilde{H}_\lambda(\x_i;z^2,w^2)\right),
$$
where 
$$
{\cal H}_{r,\lambda} (z,w):=\prod \frac{(z^{2a+1}-w^{2l+1})^r}
{(z^{2a+2}-w^{2l})(z^{2a}-w^{2l+2})}
$$
is the $(z,w)$-deformed Hook function with exponent $r$ and
$\tilde{H}_\lambda(\x_i,z,w)$ denotes the modified Macdonald symmetric
function in the variables $\x_i$.  

\begin{remark}
As mentioned in~\S\ref{case-ii}, note that unlike the setup
of~\cite{HLRV}, the integer $r$ is not necessarily even. This
introduces a choice of sign as
$$
{\cal H}_{r,\lambda} (w,z)=(-1)^{|\lambda|r}{\cal H}_{r,\lambda'} (z,w)
$$
\end{remark}

For a $k$-tuple of partitions ${\mu}=( \mu^1,\dots,\mu^k)$ of $n$ let
\begin{equation}
\label{H-defn}
\H_\mu(z,w):= (z^2-1)(1-w^2)\left\langle \Log
  \left(\Omega(z,w)\right),h_\mu\right\rangle,   
\end{equation}
where
$h_\mu=h_{\mu^1}(\x_1)\cdots h_{\mu^k}(\x_k)$ and $h_{\mu ^i}(\x_i)$ is the complete symmetric function in the variables $\x_i$.
\bigskip

For  a partition $\lambda$ denotes by $m_\lambda=m_\lambda(\x)$ the corresponding monomial symmetric function in the variables $\x=\{x_1, x_2,\dots\}$.

\begin{lemma} If $r=k=1$, we have
$$
\H_\mu(z,w)=
\begin{cases}
(z-w)m_{(1)}&  n=1\\
\frac 1{z^2+1}m_{(2)} + m_{(1^2)}  &   n=2.
\end{cases}
$$

\label{r=k=1}\end{lemma}

\begin{conjecture}
\label{comb-conj}
If  $r=1$ and $k=1$. We have $\H_\mu(z,w)=0$ for $n\geq2$.
Equivalently,
$$
(z^2-1)(1-w^2)\Log\left(  \sum_{\lambda\in \P}
\prod \frac{(z^{2a+1}-w^{2l+1})}
{(z^{2a+2}-w^{2l})(z^{2a}-w^{2l+2})}
  \tilde{H}_\lambda(\x;z^2,w^2)\right)
=(z-w)m_{(1)}(\x)+  \frac 1{z^2+1}m_{(2)}(\x) + m_{(1^2)}(\x).
$$
\label{0conj}
\end{conjecture}
(The equivalence follows from Lemma~\ref{r=k=1}.)


\subsubsection{Proof  of  Theorem~\ref{orient-thm}}
\label{orient-pf}

Assume that $K=\overline{\F}_q$ and we consider on $\G$ the
$\F_q$-structure induced by the Frobenius that raises coefficients of
matrices to their $q$-th power. We also  assume that the eigenvalues
of the conjugacy classes $C_1,\dots,C_k$ are in $\F_q$.  

As $\G$ is connected we have
$$
|\calM_\calC(\F_q)|=\frac{|\calU_\calC(\F_q)|}{|\G(\F_q)|}=\frac{\left|\left\{((A_i)_i,(X_i)_i)\in
      (\G(\F_q))^r\times \prod_{i=1}^kC_i(\F_q)\,\left|\,
        \prod_{i=1}^rA_i\sigma(A_i)
        \prod_{i=1}^kX_i=1\right\}\right.\right|}{|\G(\F_q)|}.  
$$

For each $i=1,\dots,k$, let $\mu^i=(\mu^i_1,\mu^i_2,\dots)$ be the
partition of $n$ given by the multiplicities of the eigenvalues of the
conjugacy class $\calC_i$ and let $\mu$ be the $k$-tuple
$(\mu^1,\dots,\mu^k)$. 
\begin{theorem}
The stack $\calM_\calC$ has  rational count. More precisely,
$$
|\calM_\calC(\F_q)|=\frac{q^{d_\mu/2}}{q-1}\H_\mu\left(\sqrt{q},\frac{1}{\sqrt{q}}\right).
$$
with $d_\mu=n^2(r-2+k)+2-\sum_{i,j}(\mu^i_j)^2$.
\label{maintheo1}\end{theorem}

\begin{proof}Thanks to Theorem \ref{Masstheo}, the proof is completely
  similar to that of~\cite[Theorem 5.2.1]{HLRV}. 
\end{proof}

Now Theorem~\ref{orient-thm} follows from Theorem~\ref{ratnl-count}.

\begin{remark} If $r=2h$ is even, let $\Sigma'$ be a compact Riemann
  surface of Euler characteristic $r-2$ and a subset
  $S'=\{\alpha_1',\dots,\alpha_k'\}\subset\Sigma'$. Consider the
  stacky character variety 

$$
\calM_\calC'=\left[\left\{\rho\in\Hom\big(\pi_1(\Sigma'\backslash S')\big),\G\,\left|\, \rho(z_i)\in C_i\right\}\right./\G\right].
$$
where $z_i$ is a single loop around the puncture $\alpha_i'$. Then by
Theorem~\ref{maintheo1} and~\cite[Theorem 1.2.3]{HLRV} we have 
$$
\E(\calM_\calC';x)=\E(\calM_\calC;x).
$$
\end{remark}

\subsection{Mixed Poincar\'e series in the case $r=k=1$}

\label{conj-r=k=1}
\begin{theorem} Assume that $r=k=1$. Then $\calM_\calC$ is empty unless
$n=1,2$ in which case we have
\begin{equation}
\label{small-conj}
\rmH_c(\calM_\calC;q,t)=\frac{\big(t\sqrt{q}\big)^{d_\mu}}{qt^2-1}
\mathbb{H}_\mu\left(t\sqrt{q},-\frac{1}{\sqrt{q}}\right).  
\end{equation}
\label{r=k=1evid}\end{theorem}
\begin{proof}
Consider  the solutions to the equation
\begin{equation}
\label{r=1-eqn}
A\sigma(A)=X, \qquad A\in \G,
\end{equation}
where $X:=\diag(\xi_1,\ldots,\xi_n)$. (For a calculation of the number
of solutions for arbitrary $X$ see~\cite{FG}.)

For $n=1$, we always have $A\sigma(A)=1$ and so the space
$\calM_\calC$ is nothing but the quotient stack
$[\mathbb{G}_m/\mathbb{G}_m]$ for the trivial action of $\mathbb{G}_m$
on itself. The later is isomorphic to
$\mathbb{G}_m\times {\rm B}\mathbb{G}_m$. The mixed Poincar\'e
polynomial of $\mathbb{G}_m$ is $t+qt^2$ and the mixed Poincar\'e
series of the classifying stack ${\rm B}\mathbb{G}_m$ is $1/(qt^2-1)$
\cite[9.1.1, 9.1.4]{HodgeIII}, therefore by K\"unneth formula
$$
\rmH_c(\calM_\calC;q,t)=\frac{t+qt^2}{qt^2-1},
$$
and so~\eqref{small-conj} is true when $r=k=n=1$ by Lemma \ref{r=k=1}.

Going back to Equation~\eqref{r=1-eqn}, if $A=(a_{i,j})$ this amounts to
solving the equations $\xi_ia_{j,i}=a_{i,j}$ for
$i,j=1,\ldots,n$. Hence if $\xi_i\neq 1$ then $a_{i,i}=0$. The
equations imply $\xi_i\xi_ja_{i,j}=a_{i,j}$.  If $x$ is generic and
$n>2$ we do not have $\xi_i\xi_j=1$ for $i\neq j$ or
$\xi_i=1$. Therefore the only solution is identically zero
and~\eqref{r=1-eqn} has no solution in $\G$.

  For $n=2$ we have two possibilities for generic $X$: 
 
 (i) $\xi_1=\xi_2=-1$ and $A=\left(\begin{array}{cc}0&\alpha\\-\alpha&
    0\end{array}\right)$, with $\alpha$ non-zero or 
    
   (ii) $\xi_1=\xi,\xi_2=\xi^{-1}$ for some $\xi\neq \pm
  1$ and $A=\left(\begin{array}{cc}0&\xi \alpha\\\alpha&0\end{array}\right)$, with
  $\alpha$ non-zero.

In case (i), we see from the above calculation that $\calM_\calC$ is
the quotient stack $[\mathbb{G}_m/\GL_2]$ where $\GL_2$ acts on
$\mathbb{G}_m$ by the determinant. Writing $\mathbb{G}_m$ as the
quotient $\GL_2/\SL_2$, we find that $[\mathbb{G}_m/\GL_2]$ is the
classifying stack ${\rm B}(\SL_2)$ whose mixed Poincar\'e series is
$\frac{1}{qt^2(qt^2-1)(qt^2+1)}$ by~\cite[9.1.1,
9.1.4]{HodgeIII}. Now~\eqref{small-conj} follows from Lemma~\ref{r=k=1}.

In case (ii), the stack $\calM_\calC$ is isomorphic to the stack
$[\mathbb{G}_m/{\rm T}_2]$, where the group ${\rm T}_2\subset\GL_2$
of diagonal matrices acts by the determinant. Writing $\mathbb{G}_m$
as ${\rm T}_2/{\rm T}'_2$, where ${\rm T}'_2={\rm T}_2\cap \SL_2\simeq
\mathbb{G}_m$, we find that  
$$
\rmH_c(\calM_\calC;q,t)=\frac{1}{qt^2-1},
$$
from which, together with Lemma \ref{r=k=1}, we
deduce~\eqref{small-conj} in this case. 
\end{proof}

\begin{remark}
  Note that the combinatorial Conjecture~\eqref{comb-conj} implies
  that the right-hand side of~\eqref{small-conj} is zero for $n>2$. In
  particular, the above theorem together with
  Conjecture~\eqref{comb-conj} imply that~\eqref{small-conj} is true
  for all $n$.
\end{remark}

\end{document}